\documentclass[11pt,twoside]{amsart}
\usepackage{}
\usepackage{amsthm}
\usepackage{amsmath}
\usepackage{amssymb}
\usepackage{amsfonts}
\usepackage{latexsym}
\usepackage{enumitem}
\usepackage{mathrsfs}
\usepackage[all]{xy} \SelectTips{eu}{}
\usepackage{hyperref}
\usepackage{eucal}
\usepackage{xcolor}

\newtheorem{intthm}{Theorem}[]

\newtheorem*{intque*}{Question}
\newtheorem*{intexa*}{Example}
\newcommand{\numberseries}{\bfseries}   
\newlength{\thmtopspace}                
\newlength{\thmbotspace}                
\newlength{\thmheadspace}               
\newlength{\thmindent}                  
\newtheoremstyle{bfupright head,slanted body}
{\thmtopspace}{\thmbotspace}
{\slshape}{\thmindent}{\bfseries}{.}{\thmheadspace}
{{\numberseries \thmnumber{#2\;}}\thmnote{#3}}

\newtheoremstyle{bfupright head,upright body}
{\thmtopspace}{\thmbotspace}
{\upshape}{\thmindent}{\bfseries}{.}{\thmheadspace}
{{\numberseries \thmnumber{#2\;}}\thmnote{#3}}

\newtheoremstyle{fixed bf head,slanted body}
{\thmtopspace}{\thmbotspace}{\slshape}
{\thmindent}{\bfseries}{.}{\thmheadspace}
{{\numberseries \thmnumber{#2\;}}\thmname{#1}\thmnote{ (#3)}}

\newtheoremstyle{fixed bf head,upright body}
{\thmtopspace}{\thmbotspace}{\upshape}
{\thmindent}{\bfseries}{.}{\thmheadspace}
{{\numberseries \thmnumber{#2\;}}\thmname{#1}\thmnote{ (#3)}}

\newtheoremstyle{numbered paragraph}
{\thmtopspace}{\thmbotspace}{\upshape}
{\thmindent}{\upshape}{}{\thmheadspace}
{{\numberseries \thmnumber{#2.}}}

\theoremstyle{bfupright head,slanted body}
\newtheorem{res}{}[section]             \newtheorem*{res*}{}

\theoremstyle{bfupright head,upright body}
               \newtheorem*{bfhpg*}{}

\theoremstyle{fixed bf head,slanted body}
\newtheorem{thm}[res]{Theorem}         \newtheorem*{thm*}{Theorem}
\newtheorem{prp}[res]{Proposition}      \newtheorem*{prp*}{Proposition}
\newtheorem{cor}[res]{Corollary}        \newtheorem*{cor*}{Corollary}
\newtheorem{lem}[res]{Lemma}            \newtheorem*{lem*}{Lemma}
         \newtheorem*{que*}{Question}
\theoremstyle{fixed bf head,upright body}
\newtheorem{dfn}[res]{Definition}       \newtheorem*{dfn*}{Definition}
\newtheorem{rmk}[res]{Remark}           \newtheorem*{rmk*}{Remark}
\newtheorem{exa}[res]{Example}           \newtheorem*{exa*}{Example}
           \newtheorem*{nota*}{Notation}
           \newtheorem*{setup*}{Setup}
\newtheorem{setup and notation}[res]{Setup and notation}  \newtheorem*{setup and notation*}{Setup and notation}

\theoremstyle{numbered paragraph}
\newtheorem{ipg}[res]{}

\newlength{\thmlistleft}        
\newlength{\thmlistright}       
\newlength{\thmlistpartopsep}   
\newlength{\thmlisttopsep}      
\newlength{\thmlistparsep}      
\newlength{\thmlistitemsep}     

\newcounter{eqc}
	{\end{list}}%

\newcounter{prt}
\newenvironment{prt}{\begin{list}{\upshape (\alph{prt})}%
		{\usecounter{prt}%
			\setlength{\leftmargin}{\thmlistleft}%
			\setlength{\labelwidth}{\thmlistleft}%
			\setlength{\rightmargin}{\thmlistright}%
			\setlength{\partopsep}{\thmlistpartopsep}%
			\setlength{\topsep}{\thmlisttopsep}%
			\setlength{\parsep}{\thmlistparsep}%
			\setlength{\itemsep}{\thmlistitemsep}}}%
	{\end{list}}%

\newcounter{rqm}
\newenvironment{rqm}{\begin{list}{\upshape (\arabic{rqm})}%
		{\usecounter{rqm}%
			\setlength{\leftmargin}{\thmlistleft}%
			\setlength{\labelwidth}{\thmlistleft}%
			\setlength{\rightmargin}{\thmlistright}%
			\setlength{\partopsep}{\thmlistpartopsep}%
			\setlength{\topsep}{\thmlisttopsep}%
			\setlength{\parsep}{\thmlistparsep}%
			\setlength{\itemsep}{\thmlistitemsep}}}%
	{\end{list}}%

	{\end{list}}%

\newenvironment{prf*}[1][Proof]{%
	\begin{proof}[\bf #1]
		\setcounter{equation}{0}
		}
	{\end{proof}
}

\newcommand{\pgref}[1]{\ref{#1}}

\renewcommand{\eqref}[1]{(\pgref{eq:#1})}

\newcommand{\thmcite}[2][?]{\cite[Theorem~#1]{#2}}

\newcommand{\dfncite}[2][?]{\cite[Definition~#1]{#2}}

\numberwithin{equation}{res}

\def\urltilda{\kern -.15em\lower .7ex\hbox{\~{}}\kern .04em}

\DeclareMathOperator*{\colim}{colim}

\newcommand{\AU}[1]{\mathsf{End}_{\calI}(#1)}

\newcommand{\xra}[2][]{\xrightarrow[#1]{\:#2\:}}

\newcommand{\Prj}[1]{\mathsf{Proj}(#1)}
\newcommand{\Inj}[1]{\mathsf{Inj}(#1)}

\newcommand{\Hom}[3][]{\operatorname{Hom}_{#1}(#2,#3)}

\newcommand{\sta}{\mathsf{sta}}
\newcommand{\lif}{\mathsf{lif}}
\newcommand{\eva}{\mathsf{eva}}
\newcommand{\fre}{\mathsf{fre}}
\newcommand{\ran}{\mathsf{ran}}

\newcommand{\FPP}{\mathsf{Fp}}

\newcommand{\CAT}{\mathsf{CAT}}
\newcommand{\Cat}{\mathsf{Cat}}
\newcommand{\Ab}{\mathsf{Ab}}

\newcommand{\calA}{\mathcal{A}}

\newcommand{\calC}{\mathcal{C}}
\newcommand{\calD}{\mathcal{D}}

\newcommand{\calI}{\mathcal{I}}
\newcommand{\calJ}{\mathcal{J}}

\newcommand{\calP}{\mathcal{P}}

\newcommand{\calX}{\mathcal{X}}
\newcommand{\calY}{\mathcal{Y}}

\newcommand{\scrD}{\mathscr{D}}
\newcommand{\scrR}{\mathscr{R}}

\newcommand{\Mor}{\mathsf{Mor}}
\newcommand{\Ob}{\mathsf{Ob}}
\newcommand{\lMod}{\textrm{-} \mathsf{Mod}}
\newcommand{\rMod}{\mathsf{Mod} \textrm{-}}
\newcommand{\lRep}{\textrm{-} \mathsf{Rep}}

\newcommand{\K}{\mbox{\sf ker}}
\newcommand{\coker}{\mbox{\rm coker}}

\newcommand{\op}{^{\sf op}}
\newcommand{\qis}{\simeq}
\newcommand{\C}{\mbox{\sf cok}}
\newcommand{\id}{\mathrm{id}}

\newcommand{\qra}{\xra{\qis}}
\newcommand{\cra}{\xra{\cong}}

\def\soft#1{\leavevmode\setbox0=\hbox{h}\dimen7=\ht0\advance
	\dimen7 by-1ex\relax\if t#1\relax\rlap{\raise.6\dimen7
		\hbox{\kern.3ex\char'47}}#1\relax\else\if T#1\relax
	\rlap{\raise.5\dimen7\hbox{\kern1.3ex\char'47}}#1\relax
	\else\if d#1\relax\rlap{\raise.5\dimen7\hbox{\kern.9ex
			\char'47}}#1\relax\else\if D#1\relax\rlap{\raise.5\dimen7
		\hbox{\kern1.4ex\char'47}}#1\relax\else\if l#1\relax
	\rlap{\raise.5\dimen7\hbox{\kern.4ex\char'47}}#1\relax
	\else\if L#1\relax\rlap{\raise.5\dimen7\hbox{\kern.7ex
			\char'47}}#1\relax\else\message{accent \string\soft
		\space #1 not defined!}#1\relax\fi\fi\fi\fi\fi\fi}
\makeatletter
\def\part{\@startsection{part}{1}%
\z@{.7\linespacing\@plus\linespacing}{.8\linespacing}%
{\LARGE\sffamily\centering}}

\makeatletter
\def\l@section{\@tocline{1}{2pt}{0pc}{}{}}
\makeatother
\let\oldtocpart=\tocpart
\renewcommand{\tocpart}[2]{\bf\large\oldtocpart{#1}{#2}}
\let\oldtocsection=\tocsection
\renewcommand{\tocsection}[2]{\bf\oldtocsection{#1}{#2}}

\title[Representations over diagrams of abelian categories]{Representations over diagrams of abelian categories I: Global structure and homological objects}

\date{ \today}



\thanks{Z.X. Di was partly supported by NSF of China (Grant No. 11971388) and the Scientific Research Funds of Huaqiao University; L.P. Li was partly supported by NSF of China (Grant No. 12171146); L. Liang was partly supported by NSF of China (Grant No. 12271230); N.N. Yu was partly supported by NSF of China (Grant No. 11971396).}

\author[Z.X. Di]{Zhenxing Di}
\address{Z.X. Di \ School of Mathematical Sciences, Huaqiao University, Quanzhou 362021, China}
\email{dizhenxing@163.com}

\author[L.P. Li]{Liping Li}
\address{L.P. Li \ LCSM (Ministry of Education), Department of Mathematics, Hunan Normal University, Changsha 410081, China.}
\email{lipingli@hunnu.edu.cn}

\author[L. Liang]{Li Liang}
\address{L. Liang \ Department of Mathematics, Lanzhou Jiaotong University, Lanzhou 730070, China}
\email{lliangnju@gmail.com}
\urladdr{https://sites.google.com/site/lliangnju}

\author[N.N. Yu]{Nina Yu}
\address{N.N. Yu \ School of Mathematical Sciences, Xiamen University, Xiamen 361005, China}
\email{ninayu@xmu.edu.cn}

\begin{document}

\begin{abstract}
Representations over diagrams of abelian categories unify quite a few notions appearing widely in literature such as representations of categories, presheaves of modules over categories, representations of species, etc. In this series of papers we study them systematically, characterizing special homological objects in representation category and constructing various structures (such as model structures and Wandhuasen category strcutres) on it. In the first paper we investigate the Grothendieck structure of the representation category, describe important functors and adjunction relations between them, and characterize special homological objects. These results lay a foundation for our future works.
\end{abstract}

\maketitle

\section*{Introduction}
\noindent
\noindent
\noindent
{\bf Motivation.}
The work described in this series of papers was initiated by the following result given in \cite{Sergio2015} and \cite{Olsson07}: for an affine scheme $(X,\mathcal{O}_X)$ and a poset $\calP$ of affine open subsets in $X$ ordered by inclusion, the descent property of quasi-coherent $\mathcal{O}_X$-modules produces an equivalence between the category of quasi-coherent sheaves on $X$ and the category of Cartesian $\mathcal{O}_X$-modules on $\calP$, and under a moderate condition, the second category is Grothendieck. We note that the above result can be interpreted in the framework of 2-categories. Explicitly, the small category $\calP$ can be viewed as a 2-category in a natural way and presheaves of commutative rings over $\calP$ can be viewed as contravariant 2-functors. This is a very special case of the notion of \textit{$\mathcal{I}$-diagrams of categories} with $\calI$ a small category, which appears widely in the literature under different names such as \textit{$\mathcal{I}$-indexed categories} by Johnstone \cite{Joh02} and \textit{pseudo lax functors} by Street in \cite{St72}. Diagrams of categories are also closely related to Grothendieck constructions, which provide a classical correspondence between diagrams of categories and coCartesian fibrations over the index category. This machinery was first applied to diagrams of sets by Yoneda and later developed in full generality by Grothendieck in \cite{G64}; see also \cite[I.5]{MM92}.

Representations over diagrams of categories were studied by many people under different names. For instances, they are called \textit{twisted representations} by Gothen and King in \cite{GK05}, \textit{twisted diagrams} by H\"{u}ttemann and R\"{ondigs} in \cite{HR08}, and representations of diagrams by Mozgovoy in \cite{Mozgovoy20}. Explicitly, given an $\calI$-diagram $\scrD$ of categories, a \textit{representation} $M$ over $\scrD$ is a rule to assign
\begin{itemize}
\item an object $M_i \in \scrD_i$ to any $i \in \Ob(\calI)$, and

\item a structural morphism $M_\alpha: \scrD_{\alpha}(M_i) \to M_j \in \scrD_j$ to
      any $\alpha: i \to j \in \Mor(\calI)$
\end{itemize}
such that two axioms are satisfied; see Definition \ref{DF OF R-M}. Denote by  $\scrD \lRep$ the category of all representations over $\scrD$. We notice that $\scrD \lRep$ unifies many special cases such as comma categories, module categories of Morita context rings, categories of additive functors from $\calI$ to an abelian category (which are called \textit{representations} of $\calI$ by representation theorists), categories of representations of (generalised) species and phyla studied in e.g. \cite{DR76, Gabriel1973, GLS17}. Another example comes from a work by Estrada and Virili \cite{SS2017}. Roughly speaking, given an $\calI$-diagram $R$ of associative rings\footnote{It is a covariant functor from $\calI$ to the category of associative rings, and is called a \textit{representation} of $\calI$ in that paper. To avoid possible confusions, we call it an \textit{$\calI$-diagram of associative rings}.}, we can construct an $\calI$-diagram $\overline{\scrR}$ of module categories with
\begin{itemize}
\item $\overline{\scrR}_i = R_i \lMod$, the category of left $R_i$-modules,
      for any $i \in \Ob(\calI)$ and

\item $\overline{\scrR}_\alpha = R_j \otimes_{R_i} -$
      for any $\alpha: i \to j \in \Mor(\calI)$
\end{itemize}
such that $\overline{\scrR} \lRep$ coincides with the category $R \lMod$ of left $R$-modules in the sense of \cite{SS2017}\footnote{We shall remind the reader that in our paper the abelian category $\overline{\scrR}_i$ is the left $R_i$-module category $R_i \lMod$ rather than the right $R_i$-module category $\rMod R_i$ used in \cite{SS2017}.}.

As we mentioned above, diagrams of categories and representations over them provide a uniform framework for research works in numerous areas, and people have considered them for many special cases. Thus it is desirable to study them systematically, in particular special homological objects in the category of representations, and various important structures (such as model structures and Wandhausen category structures) on the category of representations. We will do this in a series of papers. In the first paper, we will consider the following questions:
\begin{rqm}
\item describe conditions such that $\scrD \lRep$ is Grothendieck;

\item construct important functors and establish adjunction relations between them;

\item classify or characterize objects with special homological properties such as projective objects and injective objects,
\end{rqm}
whose answers lay a foundation for our later works appearing in next papers.

\vspace{0.5em}
\noindent
{\bf Abelian structure of $\scrD \lRep$.} Under the assumptions that the $\calI$-diagram $\scrD$ of abelian categories is \textit{right exact} (that is, $\scrD_\alpha$ is right exact for any $\alpha \in \Mor(\calI)$) and \textit{strict} (that is, $\scrD$ is a functor from $\calI$ to the meta 2-category of abelian categories rather than a pseudo-functor), Mozgovoy proved in \cite{Mozgovoy20} that $\scrD \lRep$ is abelian, and constructed a class of projective objects in this category; see \cite[Corollaries 2.9 and 2.13]{Mozgovoy20}. In \cite{SS2017}, Estrada and Virili showed that under some conditions the category $R \lMod$ of left $R$-modules, where $R$ is an $\calI$-diagram of associative rings, is a Grothendieck category, while it has a projective generator when $\calI$ is a poset; see \thmcite[3.18]{SS2017}. Our first main result shows that $\scrD \lRep$ is abelian or even Grothendieck under weaker conditions.

\begin{intthm}\label{thmA}
Let $\scrD$ be a right exact $\calI$-diagram of abelian categories. Then $\scrD \lRep$ is an abelian category. If further $\scrD_\alpha$ preserves small coproducts for any $\alpha \in \Mor(\calI)$ and $\scrD_i$ is a Grothendieck category $($resp., a Grothendieck category with a set of projective generators$)$ for any $i \in \Ob(\calI)$, then $\scrD \lRep$ is a Grothendieck category $($resp., a Grothendieck category with a set of projective generators$)$.
\end{intthm}

As an application, we deduce that if each $\scrD_i$ is locally finitely presented, then $\scrD \lRep$ is locally finitely presented as well; see Proposition \ref{local induce local}. Consequently, based on a general fact of Crawley-Boevey \thmcite[1.4(2)]{WCB94}, $\scrD \lRep$ is equivalent to the subcategory of flat objects in the functor category ${\sf Fun} (\FPP(\scrD \lRep)\op,\Ab)$; for details, see Corollary \ref{representation th}.

\vspace{0.5em}
\noindent
{\bf Functors and adjunctions.} Given an $\calI$-diagram $\scrD$ of abelian categories as well as a functor $G: \calI \to \calJ$, one obtains a $\calJ$-diagram $\scrD \circ G$ of abelian categories as well as a \textit{restriction} functor from $\scrD \lRep$ to $\scrD \circ G \lRep$. Moreover, we construct its left adjoint, called the \textit{induction} functor. These functors play a crucial role in the proof of Theorem \ref{thmA}.

Borrowing the idea from commutative algebra, we define \textit{prime ideals} $\mathcal{P}$ of the morphism set of $\calI$; see Subsection \ref{Sec: The stalk functor and its adjunctions}. For each prime ideal $\mathcal{P}$, we obtain a subcategory $\calI/\mathcal{P}$ as well as a $\calI/\mathcal{P}$-diagram $\scrD \circ \iota_{\mathcal{P}}$, where $\iota_{\mathcal{P}}: \calI/\mathcal{P} \to \calI$ is the inclusion functor. Under suitable conditions we can construct the \textit{lift functor} and its left adjoint, called the \textit{cokernel functor}. These functors play a crucial role for us to classify projective objects in $\scrD \lRep$.

\begin{intthm} \label{thmD}
Let $\mathcal{P}$ be a prime ideal of the morphism set of $\calI$ and $\scrD$ a right exact $\calI$-diagram of abelian categories. Then the inclusion functor $\iota_\calP: \calI/\calP \hookrightarrow \calI$ induces a functor
\[
\lif^{\mathcal{P}}: (\scrD \circ \iota_{\mathcal{P}}) \lRep \to \scrD \lRep.
\]
If furthermore $\scrD_i$ satisfies the axiom $\sf{AB3}$  for any $i \in \Ob(\calI/\mathcal{P})$ and $\scrD_\alpha$ preserves small coproducts for any $\alpha\in\Mor(\calI/\mathcal{P})$, then there exists a functor
\[
\C_{\mathcal{P}}: \scrD \lRep \to (\scrD \circ \iota_{\mathcal{P}}) \lRep
\]
which is a left adjoint of $\lif^{\mathcal{P}}$.
\end{intthm}

A special example of particular interest to us is as follows. Suppose that $\calI$ is a \textit{partially ordered category}\footnote{In the literature some people call them \textit{directed categories} or \textit{weakly directed categories}. In this paper, since we will consider \textit{direct categories} defined in \cite[Definition 5.1.1]{Ho99}, to avoid possible confusion, we call them partially ordered categories.}, that is, the relation that $i \preccurlyeq j$ if $\Hom[\calI]{i}{j}$ is nonempty defines a partial order on $\Ob(\calI)$. For a fixed object $i \in \Ob(\calI)$, $\mathcal{P}_i = \Mor(\calI) \backslash \AU i$ is a prime ideal. Applying the above theorem and imposing certain extra conditions on $\calI$, we obtain a pair of adjoint functors $\sta^i: \scrD_i \to \scrD \lRep$ and $\C_i: \scrD \lRep \to \scrD_i$, which are crucial for us to classify projective objects in $\scrD \lRep$.

Recall that a small category $\calI$ is called a \textit{direct category} if there is a functor $F: \calI \to \zeta$ where $\zeta$ is an ordinal such that $F$ sends non-identity morphisms in $\calI$ to non-identity morphisms in $\zeta$. Opposite categories of direct categories are called \textit{inverse categories}. The following result classifies projective (resp., injective) objects for direct (resp., inverse) categories.
For definitions of $ \Phi(\sf Proj_{\bullet})$ and $\Psi(\sf Inj_{\bullet})$, please see Definitions \ref{s and phi} and \ref{s and psi}.

\begin{intthm}
Let $\scrD$ be a right exact $\calI$-diagram of Grothendieck categories admitting enough projective objects, and suppose that $\scrD_i$ satisfies the axiom $\sf{AB4}^*$ for $i \in \Ob(\calI)$.
\begin{prt}
\item If $\calI$ is a direct category and $\scrD_\alpha$ preserves small coproducts for $\alpha \in \Mor(\calI)$, then
\[
\Prj{\scrD \lRep} = \Phi(\sf Proj_{\bullet}).
\]

\item If $\calI$ is an inverse category and $\scrD$ admits enough right adjoints, then
\[
\Inj{\scrD \lRep} = \Psi(\sf Inj_{\bullet}).
\]
\end{prt}
\end{intthm}

We end the introduction with the following notations and conventions.

\begin{bfhpg*}[\bf Notations and Conventions]
Throughout the paper,
\begin{itemize}

\item let $\calI$ be a skeletal small category with the set of objects $\Ob(\calI)$ and the set of morphisms $\Mor(\calI)$;

\item for $\alpha \in \Mor(\calI)$, write $s(\alpha)$ for its source and $t(\alpha)$ for its target;

\item for $i \in \Ob(\calI)$, denote by $e_i$ the identity on $i$, and denote by $\AU i$ the set of endomorphisms on $i$;

\item let $\calI_i$ be the subcategory of $\calI$ consisting of the single object $i$ and the single identity $e_i$;

\item for a quiver $Q$, denote by $Q_0$ the set of vertexes and by $Q_1$ the set of arrows;

\item by the term \emph{``subcategory''} in an abelian category, we always mean a full subcategory which is closed under isomorphisms and contains the zero object;

\item a functor between categories (resp., abelian categories) is always assumed to be covariant (resp., additive and covariant);


\item given an abelian category $\calA$, denote by $\Prj{\calA}$ (resp., $\Inj{\calA}$) the subcategory of $\calA$ consisting of projective (resp., injective) objects.
\end{itemize}
\end{bfhpg*}

\section{Diagrams of categories and representations over them}
\noindent
This section serves as an introduction to diagrams of categories and representations over them. We introduce definitions, examples, and elementary results, which will be extensively used in the rest of this paper.

\subsection{Diagrams of categories} \label{df of dia and exa}

In this subsection we introduce the definition of diagrams of categories, a central concept studied in this paper.

\begin{dfn} \label{generalised diagram}
An $\calI$-\emph{diagram of categories} is a tuple $(\scrD,\eta,\tau)$ (frequently denoted by $\scrD$ for brevity) consisting of the following data:
\begin{itemize}
\item for every $i \in \Ob(\calI)$, a category $\scrD_i$,
\item for every $\alpha: i \to j \in \Mor(\calI)$, a functor $\scrD_\alpha: \scrD_i \to \scrD_j$,
\item for every $i \in \Ob(\calI)$, a natural isomorphism $\eta_i: {\id}_{\scrD_i} \qra \scrD_{{e}_i}$, and
\item for any pair of composable morphisms $\alpha$ and $\beta$ in $\Mor(\calI)$, a natural isomorphism
\[
\tau_{\beta,\alpha}: \scrD_{\beta} \circ \scrD_{\alpha} \qra \scrD_{\beta\alpha}
\]
\end{itemize}
such that the following two axioms are satisfied:

\noindent ({Dia.1}) Given composable morphisms $i \overset{\alpha} \to j \overset{\beta} \to k \overset{\gamma} \to l \in \Mor(\calI)$, there exists a commutative diagram
\[
\xymatrix{
\scrD_{\gamma} \circ \scrD_{\beta} \circ \scrD_{\alpha} \ar[rr]^-{\id_{\scrD_\gamma} \ast \tau_{\beta,\alpha}} \ar[d]_{\tau_{\gamma,\beta} \ast \id_{\scrD_\alpha}} & & \scrD_{\gamma} \circ \scrD_{\beta\alpha} \ar[d]^{\tau_{\gamma, \beta\alpha}}\\
\scrD_{\gamma\beta} \circ \scrD_{\alpha} \ar[rr]^-{\tau_{\gamma\beta, \alpha}} & & \scrD_{\gamma\beta\alpha}
}
\]
of natural isomorphisms, where ``$\ast$'' means the Godment product of natural transformations.

\noindent ({Dia.2}) Given a morphism $i \overset{\alpha} \to j \in \Mor(\calI)$, there exists a commutative diagram
\[
\xymatrix{
 & \scrD_{\alpha} \ar@{=}[dd] \ar[dl]_{\id_{\scrD_{\alpha}} \ast \eta_i} \ar[dr]^{\eta_j \ast \id_{\scrD_{\alpha}}} & & \\
\scrD_{\alpha} \circ \scrD_{e_i} \ar[dr]_{\tau_{\alpha,e_i}} & & \scrD_{e_j} \circ \scrD_{\alpha} \ar[dl]^{\tau_{e_j, \alpha}}\\
& \scrD_{\alpha}
}
\]
of natural isomorphisms.

An $\calI$-diagram $(\scrD, \eta, \tau)$ of categories is said to be \emph{strict} if $\eta_i$ is the identity for any $i \in \Ob(\calI)$, and $\tau_{\beta,\alpha}$ is the identity for any pair of composable morphisms $\alpha$ and $\beta$ in $\Mor(\calI)$.
\end{dfn}

\begin{dfn} \label{df of mor betw diagram}
Let $(\scrD',\eta',\tau')$ and $(\scrD,\eta,\tau)$ be two $\calI$-diagrams of categories. A {\it morphism} $F$ from $\scrD'$ to $\scrD$ consists of the following data:
\begin{itemize}
\item for any $i \in \Ob(\calI)$, a functor $F_i:\scrD'_i \to \scrD_i$, and
\item for any $\alpha: i \to j \in \Mor(\calI)$, a natural isomorphism $F_\alpha: F_j \circ \scrD'_\alpha \to \scrD_\alpha \circ F_i$
\end{itemize}
such that the following two axioms are satisfied:

\noindent ({Mor.1}) Given composable morphisms $i \overset{\alpha} \to j \overset{\beta} \to k \in \Mor(\calI)$, there exists a commutative diagram
\[
\xymatrix{
F_k \circ \scrD'_\beta \circ \scrD'_\alpha \ar[rr]^-{F_\beta \ast \id_{\scrD'_\alpha}} \ar[d]_{\id_{F_k} \ast \tau'_{\beta,\alpha}} & & \scrD_\beta \circ F_j \circ \scrD'_\alpha \ar[rr]^-{\id_{\scrD_\beta} \ast F_\alpha} & & \scrD_\beta \circ \scrD_\alpha \circ F_i \ar[d]^{\tau_{\beta,\alpha} \ast \id_{F_i}}\\
F_k \circ \scrD'_{\beta\alpha} \ar[rrrr]^-{F_{\beta\alpha}} & & & & \scrD_{\beta\alpha} \circ F_i
}
\]
of natural isomorphisms.

\noindent ({Mor.2}) Given an object $i \in \Ob(\calI)$, there exists a commutative diagram
\quad\quad\quad\begin{center}
$\xymatrix{F_i \ar[rr]^{\eta_i\ast\id_{F_i}\quad} \ar[dr]_{\id_{F_i}\ast\eta'_i} &&\scrD_{e_i} \circ F_i\\
&F_i \circ \scrD'_{e_i}\ar[ur]_{F_{e_i}}}$
\end{center}
of natural isomorphisms.
\end{dfn}

\begin{rmk} \label{another explain for diagrams}
The categorically inclined reader may realize that an $\calI$-diagram of categories is exactly a pseudo lax functor from $\calI$, viewed naturally as a 2-category, to the meta 2-category $\CAT$ consisting of all categories (see Street \cite{St72}), a strict $\calI$-diagram is exactly a 2-functor from $\calI$ to $\CAT$, and a morphism between two $\calI$-diagrams is precisely a pseudo natural transformation.
\end{rmk}

\begin{dfn} \label{subdiagrams}
Let $(\scrD, \eta, \tau)$ and $(\scrD', \eta', \tau')$ be $\calI$-diagrams of categories. Then $\scrD'$ is said to be a \emph{subdiagram} of $\scrD$ if there exists a morphism $F: \scrD' \to \scrD$ such that each $F_i: \scrD'_i \to \scrD_i$ is the inclusion functor. In other words, $\scrD'$ is a subfunctor of the pseudo lax functor $\scrD$.
\end{dfn}

If $\scrD'$ is a subdiagram of $\scrD$, then it is easy to check that for $i \in \Ob(\calI)$ and composable $\alpha, \beta \in \Mor(\calI)$, the data $\scrD'_{\alpha}$, $\eta'_i$ and $\tau'_{\alpha, \beta}$ in $\scrD'$ is the restriction of the corresponding data $\scrD_{\alpha}$, $\eta_i$ and $\tau_{\alpha, \beta}$ in $\scrD$.

We call $\scrD$ an \textit{$\calI$-diagram of abelian categories} if $\scrD_i$ is an abelian category for each $i \in \Ob(\calI)$ and the functor $\scrD_\alpha$ is additive for any $\alpha \in \Mor(\calI)$.

\begin{dfn} \label{da with onough dijoints}
Let $\scrD$ be an $\calI$-diagram of abelian categories. Then $\scrD$ is said to be
\begin{prt}
\item \emph{exact} (resp., \emph{right exact}) if the functor $\scrD_\alpha$ is exact (resp., right exact) for any $\alpha \in \Mor(\calI)$;
\item \emph{admitting enough right adjoints} if the functor $\scrD_\alpha$ admits a right adjoint for any $\alpha \in \Mor(\calI)$.

\end{prt}
\end{dfn}

Given an abelian category $\calA$, if we let $\scrD_i = \calA$ for all $i \in \Ob(\calI)$ and $\scrD_\alpha = \id_{\calA}$ for all $\alpha\in\Mor(\calI)$, then $\scrD$ is called a \emph{trivial} $\calI$-diagram of $\calA$, which coincide with covaraint functors from $\calI$ to $\calA$.

\subsection{Representations over diagrams of categories} \label{part:Modules of Diagrams of small categories}
\noindent
In this subsection we introduce representations over an $\calI$-diagram $\scrD$ of categories, give various examples, and consider the abelian structure of the category of representations.

\begin{dfn} \label{DF OF R-M}
Let $(\scrD,\eta,\tau)$ be an $\calI$-diagram of categories. A \textit{representation} $M$ over $\scrD$ consists of the following data:
\begin{itemize}
\item for every $i \in \Ob(\calI)$, an object $M_i \in \scrD_i$, and
\item for every $\alpha: i \to j \in \Mor(\calI)$,
a structural morphism $M_\alpha: \scrD_{\alpha}(M_i) \to M_j \in \scrD_j$
\end{itemize}
such that the following two axioms are satisfied:

\noindent (Rep.1) Given composable morphisms
$i \overset{\alpha} \to j \overset{\beta} \to k \in \Mor(\calI)$,
there exists a commutative diagram
\[
\xymatrix{
\scrD_{\beta\alpha}(M_i) \ar[rr]^-{M_{\beta\alpha}} & & M_k\\
\scrD_{\beta}(\scrD_{\alpha}(M_i)) \ar[u]^{\tau_{\beta,\alpha}(M_i)} \ar[rr]^-{\scrD_{\beta}({M_\alpha})} & & \scrD_{\beta}(M_j) \ar[u]_{M_\beta}
}
\]
in $\scrD_k$, that is, $M_{\beta\alpha} \circ \tau_{\beta,\alpha}(M_i) = {M_\beta} \circ \scrD_{\beta}({M_\alpha})$.

\noindent (Rep.2) Given an object $i \in \Ob(\calI)$, there exists a commutative diagram
\[
\xymatrix{
M_i \ar[rr]^-{\id_{M_i}} \ar[dr]_{\eta_i(M_i)} &  &  M_i\\
 & \scrD_{e_i}(M_i) \ar[ur]_{M_{e_i}}              }
\]
in $\scrD_i$, that is, $M_{e_i} = \eta^{-1}_i(M_i)$.

A morphism $\omega: M \to M'$ between two representations $M$ and $M'$ over $\scrD$ is a family $\{ \omega_i: M_i \to M'_i\}_{i \in \Ob(\calI)}$ of morphisms such that the diagram
\[
\xymatrix{
\scrD_{\alpha}(M_i) \ar[rr]^-{\scrD_{\alpha}(\omega_i)} \ar[d]_{M_{\alpha}} & & \scrD_{\alpha}(M'_i) \ar[d]^{M'_{\alpha}}\\
M_j \ar[rr]^-{\omega_{j}} & & M'_j
}
\]
in $\scrD_j$ commutes for any $\alpha: i \to j \in \Mor(\calI)$.
\end{dfn}

Denote by $\scrD \lRep$ the category of all representations over $\scrD$. In the situation that $\scrD$ is an $\calI$-diagram of abelian categories, it is clear that $\scrD \lRep$ is an additive category with the zero object. Furthermore, if $\scrD'$ is a subdiagram of $\scrD$, then every representation over $\scrD'$ is also a representation of $\scrD$. Consequently, $\scrD' \lRep$ is a full subcategory of $\scrD \lRep$.

We give a few examples appearing in various context.

\begin{exa} \label{commma cate}
Let $F: \calC \to \calD$ be a right exact functor between abelian categories $\calC$ and $\calD$. One can construct an abelian category, denoted by $(F \downarrow \calD)$, whose objects are the morphisms $\sigma: F(C) \to D$ with $C \in \calC$ and $D \in \calD$, and morphisms from the object $\sigma: F(C) \to D$ to the object $\sigma': F(C') \to D'$ are the pair $(\alpha: C \to C', \beta: D \to D')$ of morphisms such that $\beta \circ \sigma = \sigma' \circ F(\alpha)$. Such a category is called a \emph{comma category} in the literature. Examples of comma categories include module categories over triangular matrix rings, morphism categories of abelian categories, etc.

Let $\calI$ be the free category associated to the quiver $(1 \overset{\alpha} \to 2)$. One can define a strict right exact $\calI$-diagram $\scrD$ of abelian categories with $\scrD_1 = \calC$, $\scrD_2 = \calD$ and $\scrD_\alpha = F$. It is evident that the comma category $(F \downarrow \calD)$ coincides with the category $\scrD \lRep$.
\end{exa}

\begin{exa} \label{Morita rings}
Recall that a \emph{Morita context ring} is a matrix
\[
\Lambda = \begin{bmatrix}
R & M\\
N & S
\end{bmatrix}
\]
where $R$ and $S$ are two associative rings, $M$ is an $(S, R)$-bimodule and $N$ is an $(R, S)$-bimodule, together with a morphism
$\phi: N \otimes_S M \to R$
of $(R, R)$-bimodules and a morphism
$\psi: M \otimes_R N \to S$
of $(S, S)$-bimodules. A left module over $\Lambda$ is a quadruple $(X, Y, f, g)$, where $X$ is a left $R$-module, $Y$ is a left $S$-module,
$f: M \otimes_R X \to Y$
is an $S$-module homomorphism, and
$g: N \otimes_S Y \to X$
is an $R$-module homomorphism such that certain compatibility conditions are satisfied;
please refer to \cite{G82} for details.

We observe that the module category $\Lambda \lMod$ of $\Lambda$ can be realized as $\scrD \lRep$ of a special diagram $\scrD$ of abelian categories. Explicitly, let $\calI$ be the free category associated to the quiver
\[
\xymatrix{
i \ar@/^/[r]^{\alpha} & j \ar@/^/[l]^{\beta}
}
\]
and let $\scrD_i = R \lMod$, $\scrD_j = S \lMod$, $\scrD_{\alpha} = M \otimes_R -$, and $\scrD_{\beta} = N \otimes_S -$. The reader can check that a representation over $\scrD$ is precisely a left $\Lambda$-module, so $\Lambda \lMod$ coincides with $\scrD \lRep$.
\end{exa}

\begin{exa} \label{diagram from a phylum}
Gao, K\"{u}lshammer, Kvamme and Psaroudakis introduced in \cite{GKKP20} the notion of a phylum on a quiver $Q = (Q_0, Q_1)$ as an extension of Gabriel's notion of species \cite{Gabriel1973}. Recall from \dfncite[4.1]{GKKP20} that a phylum $U$ on $Q$ consists of the following data:
\begin{itemize}
\item for a vertex $i \in Q_0$, an abelian category $\calA_i$;
\item for an arrow $\alpha: i \to j \in Q_1$, a pair of functors $(F_\alpha: \calA_i \to \calA_j, G_\alpha: \calA_j \to \calA_i)$ such that both $(F_\alpha, G_\alpha)$ and $(G_\alpha, F_\alpha)$ are adjoint pairs.
\end{itemize}
It is clear by comparing definitions that a phylum $U$ on $Q$ is a strict diagram of abelian categories admitting enough right adjoints satisfying the extra condition that each pair of adjoint functors is still an adjoint pair after switching the order. Furthermore, $U$-representations given in \dfncite[4.4]{GKKP20} coincide with representations over this diagram.
\end{exa}

\begin{exa} \label{sergio is an example}
Estrada and Virili introduced in \cite{SS2017} the notion of representations of associative rings, which are diagrams of associative rings in our sense and can be viewed as a common generalization of some important algebraic structures naturally arising in geometric contexts. In the Appendix we will show that the category of left $R$-modules in the sense of Estrada and Virili, where $R$ is a diagram of associative rings, coincides with the category $\overline{\scrR} \lRep$ for a certain diagram $\overline{\scrR}$ of module categories induced by $R$.
\end{exa}

In general the cateogry $\scrD \lRep$ is not abelian even if each $\scrD_i$ is abelian. However, under a mild assumption, we can show:

\begin{prp} \label{Rep is abelian}
Let $\scrD$ be a right exact $\calI$-diagram of abelian categories. Then $\scrD \lRep$ is an abelian category. Moreover, a sequence $M \to N \to K$ in $\scrD \lRep$ is exact if and only if $M_i \to N_i \to K_i$ is exact in $\scrD_i$ for each $i \in \Ob(\calI)$.
\end{prp}

\begin{prf*}
The kernel and cokernel can be defined componentwise, while the right exact condition guarantees that the cokernel defined in this way is indeed a representation in $\scrD \lRep$.
\end{prf*}

Colimits and limits in $\scrD \lRep$, which are defined componentwise, have been considered in \cite[Subsection 2.3]{HR08} under the condition that $\scrD$ admits enough right adjoints. The reader can see that the constructions and properties described there actually work in the slightly more general framework without this extra requirement. The following result gives sufficient conditions such that $\scrD \lRep$ satisfies $\sf{AB}$-axioms.

\begin{prp} \label{L AB 345}
Let $\scrD$ be a right exact $\calI$-diagram of abelian categories. If $\scrD_i$ satisfies the axiom $\sf{AB3^\ast}$ $($resp., $\sf{AB4^\ast})$ for any $i \in \Ob(\calI)$, then so does $\scrD \lRep$. If furthermore $\scrD_\alpha$ preserves small coproducts for any $\alpha \in \Mor(\calI)$ and $\scrD_i$ satisfies the axiom $\sf{AB3}$ $($resp., $\sf{AB4}$, $\sf AB5)$ for any $i \in \Ob(\calI)$, then so does $\scrD \lRep$.
\end{prp}

\begin{prf*}
By Proposition \ref{Rep is abelian}, $\scrD \lRep$ is an abelian category. Then the conclusion follows from the componentwise constructions of colimits and limits described in \cite{HR08} and the fact that the exactness of a sequence in $\scrD \lRep$ is completely determined by the exactness of the corresponding sequences obtained by restricting it to $\scrD_i$ for $i \in \Ob(\calI)$.
\end{prf*}

\begin{rmk}
We point out a small difference between the constructions of coproducts (colimits) and products (limits). Explicitly, when defining coproducts and colimits, we have to assume that $\scrD_{\alpha}$ preserves small coproducts for every $\alpha \in \Mor(\calI)$. However, when defining products and limits, we don't have to assume that $\scrD_{\alpha}$ preserves small products.
\end{rmk}

\section{Functors between categories of representations}
\label{Sec: The evaluation functor and its adjunctions}
\noindent
In this section we consider functors between representation categories induced by functors between index categories. As a main application of these functors, we give a sufficient criterion such that $\scrD \lRep$ is a Grothendieck category.

\begin{setup*}
Throughout this section, suppose that $(\scrD, \eta, \tau)$ is a right exact $\calI$-diagram of abelian categories unless otherwise specified, so that the category $\scrD \lRep$ is abelian by Proposition \ref{Rep is abelian}, though some results still hold without this assumption.
\end{setup*}

\subsection{The induction and restriction functors} \label{restriction and its left adjoint}

The main task of this subsection is to construct the restriction functor and its left adjoint.  We mention that this has been done by H\"{u}ttemann and R\"{o}ndigs in \cite[Subsection 2.4]{HR08} for diagrams admitting enough right adjoints, and their construction essentially works even without this extra assumption. Therefore, we omit proofs, but give detailed constructions for the convenience of the reader.

\begin{lem} \label{key detail for ext and res functors}
Let $G:\calJ \to \calI$ be a functor between small categories. Then $\scrD \circ G$ is a right exact $\calJ$-diagram with $(\scrD \circ G)_i = \scrD_{G(i)}$ for $i \in \Ob(\calJ)$ and $(\scrD \circ G)_\lambda = \scrD_{G(\lambda)}$ for $\lambda \in \Mor(\calJ)$.
\end{lem}

\begin{prf*}
It follows directly from definitions.
\end{prf*}

The restriction functor $G^*: \scrD \lRep \to (\scrD \circ G) \lRep$ is defined as follows: given an object $M$ and a morphism $\omega = \{\omega_i: M_i \to M'_i\}_{i \in \Ob(\calI)}: M \to M'$ in $\scrD \lRep$,
\begin{itemize}
\item for $j \in \Ob(\calJ)$, set $G^*(M)_j$ to be $M_{G(j)}$ in $\scrD_{G(j)}$;

\item for $\lambda: i \to j \in \Mor(\calJ)$, set the structural morphism $G^*(M)_\lambda: (\scrD \circ G)_\lambda(G^*(M)_i) \to G^*(M)_j$ to be the morphism $M_{G(\lambda)}: \scrD_{G(\lambda)}(M_{G(i)}) \to M_{G(j)}$;

\item set $G^*(\omega): G^*(M) \to G^*(M')$ to be $\{\omega_{G(i)}: M_{G(i)} \to M'_{G(i)}\}_{i \in \Ob(\calJ)}$.
\end{itemize}

\begin{lem} \label{l restriction functor along G}
The above construction gives an exact functor $G^*: \scrD \lRep \to (\scrD \circ G) \lRep$.
\end{lem}

Suppose that $\scrD_i$ satisfies the axiom $\sf AB3$ for $i \in \Ob(\calI)$ and $\scrD_\alpha$ preserves small coproducts for $\alpha \in \Mor(\calI)$. Equivalently, each $\scrD_i$ has small colimits and each $\scrD_\alpha$ preserves small colimits. We can construct a functor $G_!: (\scrD \circ G) \lRep \to \scrD \lRep$, called the \textit{induction functor} of $G$. A more conceptual construction of this functor is given in \cite[Subsection 2.4]{HR08}, where it is called a \emph{twisted left Kan extension}. To help the reader understand this abstract and complicated construction, we give here a more explicit and constructive description, but do not claim originality.
For this purpose, we introduce a special construction of colimits over morphism categories. Explicitly, for $i \in \Ob(\calI)$, let $G/i$ be the over category whose objects are morphisms in $\calI$ starting at $G(j)$ for a certain $j \in \Ob(\calJ)$ and ending at $i$, and morphisms in $G/i$ are morphisms $\beta: j \to j'$ in $\calJ$ such that the following diagram commutes:
\[
\xymatrix{
G(j) \ar[dr]_{\theta} \ar[rr]^{G(\beta)} & & G(j') \ar[dl]^{\theta'}\\
 & i.
}
\]

Given an object $N \in (\scrD \circ G) \lRep$, define $\tilde{N}: G/i \to \scrD_i$ as follows:
\begin{itemize}
\item for $\theta: G(j) \to i$, let $\tilde{N} (\theta) = \scrD_{\theta} (N_j)$;

\item for a morphism $\beta: j \to j'$ from $\theta: G(j) \to i$ to $\theta': G(j') \to i$, let $\tilde{N}(\beta)$ be the composite of the morphisms
\[
\scrD_{\theta} (N_j) = \scrD_{\theta' G(\beta)}(N_j) \cong \scrD_{\theta'} (\scrD_{G(\beta)}(N_j)) \to \scrD_{\theta'} (N_{j'}),
\]
where the last map is obtained by applying $\scrD_{\theta'}$ to the structural map $\scrD_{G(\beta)}(N_j) \to N_{j'}$.
\end{itemize}
It is easy to check that $\tilde{N}$ is indeed a functor, so one can define $\colim_{\theta \in G/i} \tilde{N}(\theta)$. Frequently, we also denote this colimit by $\colim_{\theta \in G/i} \scrD_\theta(N_\bullet)$, where $G(\bullet)$ is the source of $\theta$.

Now we define the induction functor $G_!: (\scrD \circ G) \lRep \to \scrD \lRep$ as follows. Given an object $N$ and a morphism $\sigma: N \to N'$ in $(\scrD \circ G) \lRep$,
\begin{itemize}
\item for $i \in \Ob(\calI)$, set $G_!(N)_i$ to be $\colim_{\theta \in G/i} \scrD_\theta(N_\bullet)$;

\item for $\alpha: i \to j \in \Mor(\calI)$, note that $\alpha \theta \in \Hom[\calI] {G(\bullet)}{j}$ for any $\theta \in \Hom[\calI] {G(\bullet)} i$, so there exists a canonical morphism
\[
s_{\alpha\theta}: \scrD_{\alpha\theta} (N_\bullet) \longrightarrow
\colim_{\delta \in G/j} \scrD_\delta(N_\bullet).
\]
By the universal property of colimits, we can find a unique morphism $\vartheta_\alpha$ such that the diagram
\begin{equation} \label{universal of colimit}
\xymatrix{
\scrD_\alpha (\scrD_\theta (N_\bullet)) \ar[r]^-{\tau_{\alpha,\theta} (N_\bullet)} \ar[d] & \scrD_{\alpha\theta}(N_\bullet)
\ar[d]^{}\\
\colim_{\theta \in G/i} \scrD_\alpha (\scrD_\theta(N_\bullet)) \ar[r]^-{\vartheta_\alpha} & \colim_{\delta \in G/j} \scrD_\delta(N_\bullet)
}
\end{equation}
commutes. On the other hand, by the universal property of colimits again, there exists a morphism $\chi_\alpha$ such that the diagram
\begin{equation} \label{ano universal of colimit}
\xymatrix{
\scrD_\alpha (\scrD_\theta(N_\bullet)) \ar[dr]^{\scrD_\alpha(s_\theta)} \ar[d] \\
\colim_{\theta \in G/i} \scrD_\alpha(\scrD_\theta(N_\bullet)) \ar[r]^-{\chi_\alpha} & \scrD_\alpha(\colim_{\theta \in G/i} \scrD_\theta(N_\bullet))
}
\end{equation}
commutes. Since $\scrD_\alpha$ preserves small colimits, $\chi_\alpha$ is an isomorphism. Thus we set the structural morphism $G_!(N)_\alpha$ to be the composite $\vartheta_\alpha \circ \chi^{-1}_\alpha$;

\item for any object $\theta$ in $G/i$, there exists a morphism $\scrD_\theta (\sigma_\bullet): \scrD_\theta(N_\bullet) \to \scrD_\theta(N'_\bullet)$. By the universal property of colimits, we can find a unique morphism $\omega_i$ such that the diagram
\begin{equation} \label{fina universal of colimit}
\xymatrix{
\scrD_\theta(N_\bullet) \ar[r]^-{\scrD_\theta(\sigma_\bullet)} \ar[d] & \scrD_\theta(N'_\bullet) \ar[d]\\
\colim_{\theta \in G/i} \scrD_{\theta} (N_{\bullet}) = G_!(N)_i \ar[r]^-{\omega_i} & G_!(N')_i = \colim_{\theta \in G/i} \scrD_\theta(N'_\bullet) }
\end{equation}
commutes. Define $G_!(\sigma): G_!(N) \to G_!(N')$ to be $\{\omega_i\}_{i \in \Ob(\calI)}$.
\end{itemize}

\begin{exa} \label{colimits for quivers}
Colimits appearing in the above construction seem mysterious, so let us give an explicit example for illustration. Let $Q = (Q_0, Q_1)$ be a quiver without loops (that is, arrows from a vertex to itself) and oriented cycles (that is, a nontrivial path which is of length at least two and starts and ends at the same vertex), and let $Q' = (Q_0', Q_1')$ be the subquiver obtained by removing from $Q$ a vertex $j$ and all paths through it. Since each quiver can be viewed as a category in a natural way, we obtain an inclusion functor $\iota: Q' \to Q$. Then objects in $\iota/j$ are paths $\gamma: i \to j$ such that $i \neq j$, and morphisms from $\gamma: i \to j$ to $\gamma': i' \to j$ are paths $\beta: i \to i'$ such that $\gamma = \gamma' \beta$ (here we need the assumption that $Q$ has no loops or oriented cycles to guarantee that every path from $i$ to $i'$ in $Q$ is also contained in $Q'$). It is easy to deduce the following observations:
\begin{itemize}
\item $\iota/j$ has a poset structure given by $\gamma \leqslant \gamma'$ if $\gamma = \gamma' \beta$ for a certain $\beta$;

\item each connected component of $\iota/j$ contains a unique arrow $\gamma: \bullet \to j$, which is the terminal object of this component.
\end{itemize}
Consequently, the colimit over $\iota/j$ is actually the coproduct indexed by arrows ending at $j$:
\[
\colim_{\theta \in \iota/j} \scrD_{\theta} (N_{\bullet}) = \coprod_{\theta \in Q_1(\ast, j)} \scrD_{\theta} (N_{\ast}).
\]
\end{exa}

The following result is \cite[Theorem 2.4.1]{HR08}, whose proof essentially holds even if $\scrD$ might not admit enough right adjoint.

\begin{prp} \label{l adjoint along G}
Suppose that $\scrD_i$ satisfies the axiom $\sf AB3$ for $i \in \Ob(\calI)$ and $\scrD_\alpha$ preserves small coproducts for $\alpha \in \Mor(\calI)$. Then $G_!$ defined above is a functor, and is the left adjoint of $G^*$.
\end{prp}

\begin{rmk}\label{coinduction}
When the $\calI$-diagram $\scrD$ admits enough right adjoints and each $\scrD_i$ satisfies the axiom $\sf AB3^*$, one can construct dually the right adjoint functor of $G^{\ast}$, which is called the \emph{coinduction functor} of $G$, by using the right adjoint of $\scrD_{\alpha}$ for $\alpha \in \Mor(\calI)$.
\end{rmk}

Now we apply the general results to a special case. Fix $i \in \Ob(\calI)$ and let $\calI_i$ be the subcategory of $\calI$ consisting of the single object $i$ and the single identity $e_i$. Note that there exists an obvious isomorphism between $(\scrD \circ \iota_i) \lRep$ and $\scrD_i$, where $\iota_i: \calI_i \to \calI$ is the canonical inclusion. 
We define the \emph{evaluation functor} \emph{at} $i$ to be the composite
\[
{\sf eva}^i: \scrD \lRep \overset{(\iota_i)^{\ast}} \longrightarrow (\scrD \circ \iota_i) \lRep \overset{\simeq} \longrightarrow  \scrD_i
\]
sending a representation $M$ over $\scrD$ to its ``local" value $M_i$ in $\scrD_i$. If $\scrD_j$ satisfies the axiom $\sf{AB3}$ for $j \in \Ob(\calI)$ and $\scrD_\alpha$ preserves small coproducts for $\alpha \in \Mor(\calI)$,
then by Proposition \ref{l adjoint along G}, the inclusion $\iota_i: \calI_i \to \calI$ induces a functor
\[
\fre_i: \scrD_i \overset{\simeq} \longrightarrow (\scrD \circ \iota_i) \lRep \overset{(\iota_i)_!} \longrightarrow \scrD \lRep.
\]
The functor $\fre_i$ has a very simple description. Indeed, since $\calI_i$ contains only one morphism, the over category $\iota_i/j$ for any $j \in \Ob(\calI)$ is discrete. Therefore, all colimits appearing in that construction become coproducts. That is, for $M_i \in \scrD_i$ and $j \in \Ob(\calI)$, one has
\[
(\fre_i(M_i))_j = \coprod_{\theta \in {\Hom[\calI]ij}} \scrD_{\theta} (M_i).
\]

\begin{cor} \label{l adjoint of evalution and free functor}
Let $i$ be an object in $\Ob(\calI)$. Suppose that $\scrD_j$ satisfies the axiom $\sf{AB3}$ for $j \in \Ob(\calI)$ and $\scrD_\alpha$ preserves small coproducts for $\alpha \in \Mor(\calI)$. Then $(\fre_i, \eva^i)$ is an adjoint pair.
\end{cor}

\subsection{Grothendieck structure}\label{Grothendieck structure}
\noindent
The evaluation functor and its left adjoint serve as a bridge connecting $\scrD \lRep$ and $\scrD_i$'s. In this subsection, we establish a few fundamental results on the structure of $\scrD \lRep$. In particular, local Grothendieck structures on $\scrD_i$'s can amalgamate to a Grothendieck structure on $\scrD \lRep$, and under certain conditions, $\scrD \lRep$ is locally finitely presented. Consequently, we can obtain a generalization of the classical representation theorem of Makkai and Par\'{e} \cite{MP89} (see also \cite{AdamekRosicky} and \cite{WCB94}).

\begin{lem} \label{induce generator}
Suppose that $\scrD_i$ satisfies the axiom $\sf{AB3}$ for $i \in \Ob(\calI)$ and $\scrD_\alpha$ preserves small coproducts for $\alpha \in \Mor(\calI)$.
If $\scrD_i$ has a set of generators (resp., projective generators) for $i \in \Ob(\calI)$, then so does $\scrD \lRep$.
\end{lem}

\begin{prf*}
Denote by $\mathcal{G}_i$ the set of generators of $\scrD_i$ for $i \in \Ob(\calI)$. We claim that
\[
\fre(\mathcal{G}) = \{ \fre_i(G_i) \, \mid \, G_i \in \mathcal{G}_i, i \in \Ob(\calI) \}
\]
is a set of generators of $\scrD \lRep$.

Let $\omega: M \to M'$ be a non-zero morphism in $\scrD \lRep$. Then there exists $i \in \Ob(\calI)$ such that $\omega_i: M_i \to M'_i$ is non-zero. Since $\mathcal{G}_i$ is a set of generators of $\scrD_i$, one can find a morphism $g_i: G_i \to M_i$ with $G_i \in \mathcal{G}_i$ such that $\omega_i \circ g_i \neq 0$. Clearly, the claim holds if there is a morphism $g: \fre_i(G_i) \to M$ such that $\omega \circ g \neq 0$. But this is clear. Indeed, since $(\fre_i, \eva^i)$ is an adjoint pair by Corollary \ref{l adjoint of evalution and free functor}, the morphism $g_i: G_i \to M_i = \eva^i(M)$ lifts to a unique morphism $g: \fre_i(G_i) \to M$ such that its component corresponded to $i$ is exactly $g_i$. Applying the functor $\eva^i$, one has
\[
\eva^i (\omega \circ g) = \eva^i(\omega) \circ \eva^i(g) = \omega_i \circ g_i \neq 0,
\]
so $\omega \circ g$ is nonzero as well.

In particular, if the objects in $\mathcal{G}_i$ are projective for $i \in \Ob(\calI)$, then $\fre(\mathcal{G})$ is a set of projective generators of $\scrD \lRep$ since $\fre_i$, as the left adjoint of the exact functor $\eva^i$, preserves projectives.
\end{prf*}

As an immediate consequence of Proposition \ref{L AB 345} and Lemma \ref{induce generator}, we have:

\begin{thm} \label{local grenthe induce local grenthen}
Suppose that $\scrD_\alpha$ preserves small coproducts for $\alpha \in \Mor(\calI)$. If $\scrD_i$ is a Grothendieck category (resp., a Grothendieck category with a set of projective generators) for $i \in \Ob(\calI)$, then so is $\scrD \lRep$.
\end{thm}

\begin{rmk} \label{generator for R-represetations}
Suppose that $\scrD_i$ is a Grothendieck category admitting a set of projective generators for $i\in\Ob(\calI)$ and $\scrD_\alpha$ preserves small coproducts for $\alpha \in \Mor(\calI)$. Then the Grothendieck category $\scrD \lRep$ admits enough projecitves. By the proof of Lemma \ref{induce generator}, it is easy to see that any projective object in $\scrD \lRep$ is isomorphic to a direct summand of a direct sum of members in the family
\[
\{ \fre_i(P_i) \, \mid \, i \in \Ob(\calI) \text{ and } P_i \text{ is projective in } \scrD_i \}.
\]
\end{rmk}

Let $\calA$ be an abelian category satisfying the axiom $\sf{AB3}$. Recall that an object $X$ in $\calA$ is said to be \emph{finitely presented} provided that
the representable functor $\Hom[\calA]{X}{-}$ commutes with filtered colimits. Denote the full subcategory of finitely presented objects by $\FPP (\calA)$.
Recall that $\calA$ is \emph{locally finitely presented} if $\FPP (\calA)$ is skeletally small and every object in $\calA$ is a filtered colimit of finitely presented objects, or equivalently, $\calA$ possesses a set of finitely presented generators; see \cite[Theorem 1.11]{AdamekRosicky}\footnote{By \cite[Theorem 1.11]{AdamekRosicky}, $\calA$ is locally finitely presented if and only if it possesses a set of strong generators formed by finitely presentable objects. But note that in an abelian category, any set of generators is a set of strong generators.}. The following proposition tells us that the locally finitely presented property of each $\scrD_i$ can also be amalgamated to the locally finitely presented property of $\scrD \lRep$.

\begin{prp} \label{local induce local}
Suppose that $\scrD_\alpha$ preserves small coproducts for $\alpha \in \Mor(\calI)$. If $\scrD_i$ is a locally finitely presented for $i \in \Ob(\calI)$, then so is $\scrD \lRep$.
\end{prp}

\begin{prf*}
By Lemma \ref{induce generator}, it is enough to show that $\fre_i$ preserves finitely presented object in $\scrD_i$ for $i \in\Ob(\calI)$. Let $N_i$ be a finitely presented object in $\scrD_i$ and $(M^x, f^{yx})$ be a filtered direct system of objects in $\scrD \lRep$. Then we have
\begin{align*}
\Hom[{\scrD \lRep}]{\fre_i(N_i)}{\colim \, M^x}
& \cong \Hom[\scrD_i]{N_i}{\eva^i(\colim \, M^x)}\\
& = \Hom[\scrD_i]{N_i}{\colim \, M^x_i}\\
& \cong \colim \Hom[\scrD_i]{N_i}{M^x_i}\\
& \cong \colim \Hom[{\scrD \lRep}]{\fre_i(N_i)}{M^x},
\end{align*}
where the first isomorphism holds by Corollary \ref{l adjoint of evalution and free functor}, the second equality holds since $\eva^i(\colim \, M^x) = (\colim \, M^x)_i$ and the filtered colimits in $\scrD \lRep$ are computed componentwise, the third isomorphism holds as $N_i$ is a finitely presented object in $\scrD_i$, and the last one holds by Corollary \ref{l adjoint of evalution and free functor} again.
\end{prf*}

The next result, following from \thmcite[1.4(2)]{WCB94} and Proposition \ref{local induce local}, gives the representation theorem for $\scrD \lRep$. Recall from \cite{Oberst70} that an object in the functor category ${\sf Fun}(\calC\op,\Ab)$ (where $\calC$ is a skeletally small additive category) is called \emph{flat} if it is a colimit of representable functors.

\begin{cor} \label{representation th}
Suppose that $\scrD_\alpha$ preserves small coproducts for $\alpha \in \Mor(\calI)$. If $\scrD_i$ is locally finitely presented for $i \in \Ob(\calI)$, then $\scrD \lRep$ is equivalent to the subcategory of flat objects in the functor category ${\sf Fun}(\FPP(\scrD \lRep)\op,\Ab)$.
\end{cor}

\subsection{The lift and cokernel functors} \label{Sec: The stalk functor and its adjunctions}

\noindent A nonempty subset $\calP$ of $\Mor(\calI)$ is called a \emph{two-sided ideal} of $\calI$ if for any $\alpha \in \Mor(\calI)$ and any morphism $\beta \in \mathcal{P}$, one has that $\alpha\beta$ or $\beta\alpha$, whenever composable, is always contained in $\mathcal{P}$. The two-sided ideal $\calP$ of $\calI$ is said to be \emph{prime} if the converse statement holds, that is, one has $\alpha \in \calP$ or $\beta \in \calP$ whenever $\alpha \beta \in \calP$.
From the above definition, it is easy to see that a two-sided ideal $\calP$ of $\Mor(\calI)$ is prime if and only if $\Mor(\calI) \backslash \calP$ is closed under compositions of morphisms.

Given a prime ideal $\calP$ of $\calI$, we can construct a subcategory $\calI/\calP$ of $\calI$ as follows:
\begin{itemize}
\item $i \in \Ob(\calI)$ is also an object in $\Ob(\calI/\calP)$ if $ e_i \notin \calP$.

\item $\alpha: i \to j  \in \Mor(\calI)$ is contained in $\Mor(\calI/\calP)$ if $\alpha \notin \calP$.\footnote{This is well defined since in this case $e_i$ and $e_j$ can not be in $\calP$ as $\calP$ is a two-sided ideal, and hence $i$ and $j$ are objects in $\Ob(\calI/\calP)$.}
\end{itemize}
Note that for any pair of composable morphisms $\alpha$ and $\beta$ in $\Mor(\calI)$ such that neither $\alpha$ nor $\beta$ is in $\calP$, their composite $\beta\alpha$ is not in $\calP$ since $\calP$ is prime by assumption.  Therefore, $\beta\alpha$ belongs to $\Mor(\calI/\calP)$, so  $\calI/\calP$ is indeed a subcategory. Furthermore, one has $\Mor(\calI/\calP) = \Mor(\calI) \backslash \calP$, the complement set of $\calP$ in $\Mor(\calI)$.

Let $\iota_\calP: \calI/\calP \hookrightarrow \calI$ be the inclusion functor. According to Lemma \ref{key detail for ext and res functors}, the $\calI$-diagram $\scrD$ of abelian categories induces an $\calI/\calP$-diagram $\scrD \circ \iota_\calP$ of abelian categories. We construct a \emph{lift functor} $\lif^{\calP}$ from $(\scrD \circ \iota_\calP) \lRep$ to $\scrD \lRep$, which roughly speaking, is obtained by adding zeroes. Explicitly, given an object $M$ and a morphism $\sigma: M \to M'$ in $(\scrD \circ \iota_\calP) \lRep$, define

\begin{itemize}
\item for $i \in \Ob(\calI)$,
\[
\lif^{\calP}(M)_i =
\begin{cases}
M_i, & \text{if } i \in \Ob(\calI/\calP),\\
0, & \text{otherwise};
\end{cases}
\]

\item for $\alpha: i \to j \in \Mor(\calI)$,
\[
\lif^{\calP}(M)_\alpha =
\begin{cases}
M_\alpha: \scrD_\alpha(M_i) \to M_j, & \text{if } \alpha \in \Mor(\calI/\calP),\\
0, & \text{otherwise};
\end{cases}
\]

\item for $i \in \Ob(\calI)$,
\[
\lif^{\calP}(\sigma)_i =
\begin{cases}
\sigma_i, & \text{if } i \in \Ob(\calI/\calP),\\
0,  & \text{otherwise.}
\end{cases}
\]
\end{itemize}

\begin{lem} \label{The lift functor}
Let $\calP$ be a prime ideal of $\calI$. Then $\lif^{\calP}: (\scrD \circ \iota_\calP) \lRep \to \scrD \lRep$ is an exact functor.
\end{lem}

\begin{prf*}
We show that $\lif^{\calP}$ sends objects to objects and morphisms to morphisms. Other axioms of functors can be verified routinely. To show that $\lif^{\calP}(M)$ is an object in $\scrD \lRep$, we have to check that $\lif^{\calP}(M)$ satisfies  the axioms (Rep.1) and (Rep.2) in Definition \ref{DF OF R-M}.

For the axiom (Rep.1), we want to prove the equality
\[
\lif^{\calP}(M)_{\beta\alpha} \circ \tau_{\beta,\alpha} (\lif^{\calP}(M)_i) = \lif^{\calP}(M)_\beta \circ \scrD_\beta(\lif^{\calP}(M)_\alpha)
\]
for any pair $i \overset{\alpha} \to j \overset{\beta} \to k$ of composable morphisms in $\Mor(\calI)$. We have four cases.
\begin{prt}
\item If $\alpha \in \Mor(\calI/\calP)$ and $\beta \in \Mor(\calI/\calP)$, then $\alpha \notin \calP$ and $\beta \notin\calP$, so $\beta \alpha \in \Mor(\calI/\calP)$ since $\Mor(\calI/\calP) = \Mor(\calI) \backslash \calP$ is closed under composition of morphisms. In this case the desired equality holds as $M$ satisfies the axioms (Rep.1).

\item If $\alpha \in \Mor(\calI/\calP)$ but $\beta \notin \Mor(\calI/\calP)$, then $\beta \in \calP$ and $\lif^{\calP}(M)_\beta = 0$, so $\lif^{\calP}(M)_\beta \circ \scrD_\beta(\lif^{\calP}(M)_\alpha) = 0$. On the other hand, since $\calP$ is an ideal, $\beta\alpha \in \calP$. Hence $\beta\alpha \notin \Mor(\calI/\calP)$,
    which implies that $\lif^{\calP}(M)_{\beta\alpha} = 0$. Consequently,
    $\lif^{\calP} (M)_{\beta\alpha} \circ \tau_{\beta,\alpha}(\lif^{\calP}(M)_i) = 0$,
    as desired.

\item If $\alpha \notin \Mor(\calI/\calP)$ but $\beta \in \Mor(\calI/\calP)$, then one can verify the equality as in the case (b).

\item If $\alpha \notin \Mor(\calI/\calP)$ and $\beta \notin \Mor(\calI/\calP)$, then $\lif^{\calP}(M)_\alpha = 0$ and $\lif^{\calP}(M)_\beta = 0$, so
    $$\lif^{\calP} (M)_\beta \circ \scrD_\beta(\lif^{\calP} (M)_\alpha) = 0.$$
    On the other hand, note that $\beta\alpha \notin \Mor(\calI/\calP)$ because $\calP$ is closed under composition. It follows that $\lif^{\calP}(M)_{\beta\alpha} = 0$, which implies the desired equality.
\end{prt}

For the axiom (Rep.2), we have to prove the equality
\[
  \lif^{\calP}(M)_{e_i} \circ \eta_i(\lif^{\calP}(M)_i) = \id_{\lif^{\calP}(M)_i}
\]
for any $i \in \Ob(\calI)$. Indeed, if $i \in \Ob(\calI/\calP)$, then $e_i \in \Mor(\calI/\calP)$, so we have
\[
\lif^{\calP}(M)_{e_i} \circ \eta_i(\lif^{\calP}(M)_i) = M_{e_i} \circ \eta_i(M_i) = \id_{M_i} = \id_{\lif^{\calP}(M)_i},
\]
where the second equality holds as $M$ satisfies the axiom (Rep.2). In the case where $i \notin \Ob(\calI/\calP)$, we have $e_i \notin \Mor(\calI/\calP)$ and $\lif^{\calP}(M)_i = 0$, so
\[
\lif^{\calP}(M)_{e_i} \circ \eta_i(\lif^{\calP}(M)_i) = 0 = \id_{\lif^{\calP}(M)_i}.
\]
Thus, $\lif^{\calP}(M)_\alpha$ satisfies the axiom (Rep.2) as well.

To verify that $\lif^{\calP}(\sigma)$ is a morphism in $\scrD \lRep$, we have to show
\[
\lif^{\calP}(\sigma)_j \circ \lif^{\calP}(M)_{\alpha} = \lif^{\calP}(M')_{\alpha} \circ \scrD_\alpha (\lif^{\calP}(\sigma)_i)
\]
for all $\alpha: i \to j \in \Mor(\calI)$. Indeed, if $\alpha \in \Mor(\calI/\calP)$, then both $i$ and $j$ are in $\Ob(\calI/\calP)$. Hence
$\lif^{\calP}(\sigma)_j \circ \lif^{\calP}(M)_{\alpha} = \sigma_j \circ M_\alpha = M'_\alpha \circ \scrD_\alpha(\sigma_i) = \lif^{\calP}(M')_{\alpha} \circ \scrD_\alpha( \lif^{\calP}(\sigma)_i)$.
If $\alpha \notin \Mor(\calI/\calP)$, then both $\lif^{\calP}(M)_{\alpha}$ and $\lif^{\calP}(M')_{\alpha}$ are zero. The desired equality holds clearly in this case.
\end{prf*}


Let $\calP$ be a prime ideal of $\calI$.
In what follows, for any $i \in \Ob(\calI)$, denote
\[
\calP(\bullet, i) = \{\theta \in \Mor(\calI)\, \mid \, \theta \in \calP \text{ with target } i\}.
\]
Suppose that $\scrD_i$ satisfies the axiom $\sf{AB3}$ for $i \in \Ob(\calI/\calP)$ and $\scrD_{\alpha}$ preserves small coproducts for $\alpha \in \Mor(\calI/\calP)$. For any $M \in \scrD \lRep$, by the universal property of colimits, there exists a unique morphism ${\varphi}_i^M$ such that for every $\theta \in \calP(\bullet, i)$, the diagram
\begin{equation} \label{map-def}
\xymatrix{
\scrD_\theta(M_{s(\theta)}) \ar[d] \ar[drr]^{M_\theta}\\
\colim_{\theta \in \calP(\bullet, i)} \scrD_\theta(M_{s(\theta)}) \ar[rr]^{\ \ \ \ \ \ {\varphi}_i^M} && M_i
}
\end{equation}
in $\scrD_i$ commutes.
We use the above commutative diagram to define the left adjoint
\[
\C_\calP: \scrD\lRep \to (\scrD \circ \iota_\calP)\lRep
\]
of the lift functor. Explicitly, given an object $M$ and a morphism $\omega = \{\omega_i\}_{i \in \Ob(\calI)}: M \to M'$ in $\scrD \lRep$,
\begin{itemize}
\item for $i \in \Ob(\calI/\calP)$, set $\C_\calP(M)_i$ to be $\coker({\varphi}_i^M)$. Explicitly,
\[
\C_\calP(M)_i = M_i/\sum_{\theta \in \calP(\bullet, i)} M_\theta(\scrD_\theta(M_{s(\theta)})).
\]
We mention that the sum appearing in the above equality is precisely the image of $\varphi_i^M$.

\item For $\alpha: i \to j \in \Mor(\calI/\calP)$, one can obtain a composite of morphisms
\[\quad\quad\quad
M_{\alpha} \big{(} \scrD_{\alpha} (\sum_{\theta \in \calP(\bullet, i)} M_\theta(\scrD_\theta(M_{s(\theta)}))) \big{)} \rightarrow \sum_{\theta \in \calP(\bullet, i)} M_{\alpha \theta} (\scrD_{\alpha \theta} (M_{s(\theta)})) \rightarrow
\sum_{\delta \in \calP(\bullet, j)} M_{\delta} (\scrD_{\delta} (M_{s(\delta)}))
\]
where the first one is obtained by applying a suitable natural transformation in (Rep.1) and the second one is an actual inclusion. This composite of morphisms gives rise to a morphism $\varrho_{\alpha}$ making the following diagram
\begin{equation} \label{third C}
\xymatrix{
\scrD_\alpha(\colim_{\theta \in \calP(\bullet, i)} \scrD_\theta(M_{s(\theta)})) \ar[r]^-{\scrD_\alpha(\varphi_i^M)} & \scrD_\alpha (M_i) \ar[r]^-{\scrD_\alpha(\pi_i^M)} \ar[d]^{M_\alpha} & \scrD_\alpha(\C_\calP(M)_i) \ar[r] \ar[d]^{\varrho_\alpha} & 0\\
\colim_{\delta \in \calP(\bullet, j)} \scrD_\delta(M_{s(\delta)}) \ar[r]^-{\varphi_j^M} & M_j \ar[r]^-{\pi_j^M} & \C_\calP(M)_j \ar[r] & 0
}
\end{equation}
commutes. We then define $\C_\calP(M)_\alpha$ to be $\varrho_\alpha$.
\item For $i \in \Ob(\calI/\calP)$, one has
\begin{align*}
\pi_i^{M'} \circ \omega_i \circ \varphi_i^M \circ s_\theta^{M} &=
\pi_i^{M'} \circ \omega_i \circ M_\theta \\
&= \pi_i^{M'} \circ M'_\theta \circ \scrD_\theta(\omega_{s(\theta)})\\
&= \pi_i^{M'} \circ \varphi_i^{M'} \circ s_\theta^{M'} \circ \scrD_\theta(\omega_{s(\theta)})
= 0,
\end{align*}
where the first and third equalities follow from (\ref{map-def}).
By the universal property of colimts, $\pi_i^{M'} \circ \omega_i \circ \varphi_i^M = 0$. Therefore, by the universal property of cokernels, we can find a unique morphism $\C_\calP(\omega)_i$ such that the diagram
\begin{equation} \label{FOUR C}
\xymatrix{
\colim_{\theta \in \calP(\bullet, i)} \scrD_\theta(M_{s(\theta)}) \ar[r]^-{\varphi_i^M} & M_i \ar[r]^-{\pi_i^M} \ar[d]^{\omega_i} & \C_\calP(M)_i \ar[r] \ar[d]^{\C_\calP(\omega)_i} & 0\\
\colim_{\theta \in \calP(\bullet, i)} \scrD_\theta(M'_{s(\theta)}) \ar[r]^-{\varphi_i^{M'}} & M'_i \ar[r]^-{\pi_i^{M'}} & \C_\calP(M')_i \ar[r] & 0
}
\end{equation}
commutes.
\end{itemize}

\begin{thm} \label{coker adjoint pair}
Let $\calP$ be a prime ideal of $\calI$.
Suppose that $\scrD_i$ satisfies the axiom $\sf{AB3}$ for $i \in \Ob(\calI/\calP)$ and $\scrD_\alpha$ preserves small coproducts for $\alpha \in \Mor(\calI/\calP)$.
Then $\C_\calP: \scrD\lRep \to (\scrD \circ \iota_\calP)\lRep$ defined above is a functor, and is the left adjoint of $\lif^{\calP}$.
\end{thm}

\begin{prf*}
To establish the first statement, it is enough to show that $\C_\calP$ sends objects to objects and morphisms to morphisms, and other axioms of functors are clear since the above construction of $\C_{\calP}$ is functorial.

\begin{rqm}
\item  $\C_\calP(M)$ is an object in $(\scrD \circ \iota_\calP) \lRep$. We need to verify the axioms (Rep.1) and (Rep.2) in Definition \ref{DF OF R-M}. For any pair $i \overset{\alpha} \to j \overset{\beta} \to k$ of morphisms in $\Mor(\calI/\calP)$, we have equalities
\begin{align*}
&\C_\calP(M)_\beta \circ \scrD_\beta(\C_\calP(M)_\alpha) \circ \scrD_\beta(\scrD_\alpha(\pi_i^M))\\
&= \C_\calP(M)_\beta \circ \scrD_\beta(\pi_j^M) \circ \scrD_\beta(M_\alpha) & \text{by (\ref{third C})}\\
&= \pi_k^M \circ M_\beta \circ \scrD_\beta(M_\alpha) & \text{by (\ref{third C})}\\
&= \pi_k^M \circ M_{\beta\alpha} \circ \tau_{\beta,\alpha}(M_i) & \text{by (Rep.1)}\\
&= \C_\calP(M)_{\beta\alpha} \circ \scrD_{\beta\alpha}(\pi_i^M) \circ \tau_{\beta,\alpha}(M_i) & \text{by (\ref{third C})}\\
&= \C_\calP(M)_{\beta\alpha} \circ \tau_{\beta,\alpha}(\C_\calP(M)_i) \circ \scrD_\beta(\scrD_\alpha(\pi_i^M)),
\end{align*}
where the last equality follows from the commutative diagram obtained by applying  $\tau_{\beta,\alpha}$ to $\pi_i^M$. Since both $\scrD_\beta$ and $\scrD_\alpha$ are right exact and $\pi_i^M$ is an epimorphism, we conclude that $\scrD_\beta(\scrD_\alpha(\pi_i^M))$ is also an epimorphism. Thus we have
\[
\C_\calP(M)_\beta \circ \scrD_\beta(\C_\calP(M)_\alpha)=\C_\calP(M)_{\beta\alpha} \circ \tau_{\beta,\alpha} (\C_\calP(M)_i).
\]
Hence $\C_\calP(M)$ satisfies the axiom (Rep.1).

For $i \in \Ob(\calI/\calP)$, we have equalities
\begin{align*}
\C_\calP(M)_{e_i} \circ \eta_i(\C_\calP(M)_i) \circ \pi_i^M
&= \C_\calP(M)_{e_i} \circ \scrD_{e_i}(\pi_i^M) \circ \eta_i(M_{e_i})\\
&= \pi_i^M \circ M_{e_i} \circ \eta_i(M_{e_i}) & \text{by (\ref{third C})}\\
&= \pi_i^M,
\end{align*}
where the first equality holds by the commutative diagram obtained by applying $\eta_i$ to $\pi_i^M$. Since $\pi_i^M$ is an epimorphism, one has
\[
\C_\calP(M)_{e_i} \circ \eta_i(\C_\calP(M)_i) = \id_{\C_\calP(M)_i},
\]
which is exactly the axiom (Rep.2).
\item $\C_\calP(\omega) = \{\C_\calP(\omega)_i\}_{i \in \Ob(\calI/\calP)}$ is a morphism. We must show the equality
\[
\C_\calP(\omega)_j \circ \C_\calP(M)_\alpha = \C_\calP(M')_\alpha \circ \scrD_\alpha(\C_\calP(\omega)_i)
\]
for any $\alpha: i \to j \in \Mor(\calI/\calP)$.  But we have
\begin{align*}
 \C_\calP(\omega)_j \circ \C_\calP(M)_\alpha \circ \scrD_\alpha(\pi_i^M)
&= \C_\calP(\omega)_j \circ \pi_j^M \circ M_\alpha & \text{by (\ref{third C})}\\
&= \pi_j^{M'} \circ \omega_j \circ M_\alpha & \text{by (\ref{FOUR C})}\\
&= \pi_j^{M'} \circ M'_\alpha \circ \scrD_\alpha(\omega_i)\\
&= \C_\calP(M')_\alpha \circ \scrD_\alpha(\pi_i^{M'} ) \circ \scrD_\alpha(\omega_i) & \text{by (\ref{third C})}\\
&= \C_\calP(M')_\alpha \circ \scrD_\alpha(\C_\calP(\omega)_i) \circ \scrD_\alpha(\pi_i^M), & \text{by (\ref{FOUR C})}
\end{align*}
where the third equality holds as $\omega$ is a morphism in $\scrD \lRep$. Since $\scrD_\alpha$ is right exact and $\pi_i^M$ is an epimorphism, it follows that $\scrD_\alpha (\pi_i^M)$ is also an epimorphism, and the desired equality follows.

We have established the first statement. Now we prove the second one. Let $M$ be an object in $\scrD \lRep$ and $N$ an object in $(\scrD \circ \iota_\calP) \lRep$. We construct a pair of natural maps
\[
u: \Hom[{(\scrD \circ \iota_\calP)\lRep}]{\C_\calP(M)}{N} \rightleftarrows \Hom[\scrD \lRep]{M}{\lif^{\calP}(N)} :v
\]
which are inverse to each other. Note that there exists an exact sequence
\[
\colim_{\theta \in \calP(\bullet, i)} \scrD_\theta(M_{s(\theta)}) \overset{\varphi_i^M} \longrightarrow  M_i  \overset{\pi^M_i} \longrightarrow \C_\calP(M)_i \to 0
\]
in $\scrD_i$ for $i \in \Ob(\calI/\calP)$.
\item Define the map $u$. Let $\sigma = \{\sigma_i\}_{i \in \Ob(\calI/\calP)}: \C_\calP(M) \to N$ be a morphism in $(\scrD \circ \iota_\calP) \lRep$. Then for $i \in\Ob(\calI)$, define
\begin{equation*}
u(\sigma)_i =
\begin{cases}
\sigma_i \circ \pi^M_i, & \text{if } i \in \Ob(\calI/\calP),\\
0, & \text{otherwise}
\end{cases}
\end{equation*}
by noting that $\lif^{\calP} (N)_i = N_i$ if $i \in \Ob(\calI/\calP)$ and $N_i = 0$ otherwise.

To show that $u(\sigma) = \{u(\sigma)_i\}_{i \in \Ob(\calI)}$ is a morphism from $M$ to $\lif^{\calP}(N)$, we verify the equality
\[
u(\sigma)_j \circ M_\alpha = \lif^{\calP}(N)_\alpha \circ \scrD_\alpha(u(\sigma)_i) \quad \quad \quad \quad (\sharp)
\]
for any $\alpha: i \to j \in \Mor(\calI)$ case by case.

\begin{prt}
\item If $j \notin \Ob(\calI/\calP)$, then $\lif^{\calP}(N)_j = 0$, and the equality $(\sharp)$ holds.

\item If $j \in \Ob(\calI/\calP)$ but $i \notin \Ob(\calI/\calP)$, then $e_i \in \calP$, so $\alpha \in \calP(\bullet, j)$. Consequently, $M_\alpha = {\varphi}_j^M \circ s_\alpha^{M}$, and one has
\[
u(\sigma)_j \circ M_\alpha = \sigma_j \circ \pi^M_j \circ M_\alpha = \sigma_j \circ \pi^M_j \circ \varphi_j^M \circ s_\alpha^{M} = 0.
\]
On the other hand, since $u(\sigma)_i = 0$ by definition, we have $\lif^{\calP}(N)_\alpha \circ \scrD_\alpha(u(\sigma)_i) = 0$.
Thus the equality $(\sharp)$ holds, too.

\item If $j \in \Ob(\calI/\calP)$, $i \in \Ob(\calI/\calP)$ but $\alpha \notin \Mor(\calI/\calP)$, then one can check the equality by an argument similar to that of the previous case.

\item If $j \in \Ob(\calI/\calP)$, $i \in \Ob(\calI/\calP)$ and $\alpha \in \Mor(\calI/\calP)$, then one has
\begin{align*}
 \lif^{\calP}(N)_\alpha \circ \scrD_\alpha( u(\sigma)_i)
&= N_\alpha \circ \scrD_\alpha(\sigma_i) \circ \scrD_\alpha(\pi^M_i)\\
&= \sigma_j \circ \C_\calP(M)_\alpha \circ \scrD_\alpha(\pi^M_i)\\
&= \sigma_j \circ \pi^M_j \circ M_\alpha\\
&= u(\sigma)_j \circ M_\alpha,
\end{align*} where the first equality holds by the definitions of $\lif^{\calP}(N)_\alpha$ and $u(\sigma)_i$.
\end{prt}

\item Define the map $v$. Let $\omega = \{\omega_i\}_{i \in \Ob(\calI)}: M \to \lif^{\calP}(N)$ be a morphism in $\scrD \lRep$. Then for $i \in \Ob(\calI/\calP)$ and
$\theta \in \calP(\bullet, i)$, one has
\[
\omega_i \circ \varphi_i^M \circ s_\theta^{M} = \omega_i \circ M_\theta = \lif^{\calP}(N)_\theta \circ \scrD_\theta(\omega_{s(\theta)}) = 0,
\]
where the first equality holds by (\ref{map-def}), and the last equality holds since $\theta \notin \Mor(\calI/\calP)$. Therefore, $\omega_i \circ {\varphi}_i^M = 0$ by the universal property of colimts, and hence we can find a unique morphism $v(\omega)_i$ such that the diagram
\begin{equation*} \label{adjoint 1}
\tag{\ref{coker adjoint pair}.1}
\xymatrix{
\colim_{\theta \in \calP(\bullet, i)} \scrD_\theta(M_{s(\theta)})\ar[r]^-{\varphi_i^M} & M_i \ar[r]^-{\pi^M_i} \ar[d]_{\omega_i} & \C_\calP(M)_i \ar[dl]^{v(\omega)_i} \ar[r] & 0\\
& \lif^{\calP}(N) = N_i
}
\end{equation*}
commutes.

We show that $v(\omega) = \{v(\omega)_i\}_{i \in \Ob(\calI/\calP)}$ is a morphism from $\C_\calP(M)$ to $N$ by checking the equality
\[
v(\omega)_j \circ \C_\calP(M)_\alpha = N_\alpha \circ \scrD_\alpha(v(\omega)_i)
\]
for any $\alpha: i \to j \in \Mor(\calI/\calP)$. Indeed, one has
\begin{align*}
v(\omega)_j \circ \C_\calP(M)_\alpha \circ \scrD_\alpha(\pi^M_i)
&= v(\omega)_j \circ \pi^M_j \circ M_\alpha & \text{by (\ref{third C})}\\
&= \omega_j \circ M_\alpha & \text{by (\ref{adjoint 1})}\\
&= \lif^{\calP}(N)_\alpha \circ \scrD_\alpha(\omega_i)\\
&= N_\alpha \circ  \scrD_\alpha(v(\omega)_i) \circ \scrD_\alpha(\pi^M_i) &  \text{by (\ref{adjoint 1}).}
\end{align*}
Since $\scrD_\alpha(\pi^M_i)$ is an epimorphism, the desired equality follows.
\item $u$ and $v$ are inverse to each other. Clearly, $u$ and $v$ are natural with respect to $M$ and $N$. Let $i$ be an object in $\Ob(\calI)$. If $i \notin \Ob(\calI/\calP)$, then $u(v(\omega))_i = 0$ by definition. Note that $\lif^{\calP}(N)_i = 0$ in this case, so $\omega_i = 0$. Therefore, $u(v(\omega))_i = \omega_i$. If $i \in \Ob(\calI/\calP)$, then $u(v(\omega))_i = v(\omega)_i \circ \pi^M_i = \omega_i$, where the the second equality holds by (\ref{adjoint 1}). Consequently, we always have $uv(\omega)=\omega$. On the other hand, suppose that $i$ is an object in $\Ob(\calI/\calP)$. Then one has $v(u(\sigma))_i \circ \pi^M_i = u(\sigma)_i = \sigma_i \circ \pi^M_i$, where the first equality holds by (\ref{adjoint 1}). By the universal property of cokernels, we conclude that $v(u(\sigma))_i = \sigma_i$, and hence $vu(\sigma) = \sigma$.
\end{rqm}
This finishes the proof.
\end{prf*}

In the following, we apply the general results established in the above to a special type of index categories $\calI$: partially ordered categories. The relation $\preccurlyeq$ on the set $\Ob(\calI)$ such that $i \preccurlyeq j$ if $\Hom[\calI]i{j} \neq \emptyset$ is clearly reflexive and transitive, but in general not anti-symmetric. If furthermore it is anti-symmetric, then $\Ob(\calI)$ becomes a poset with respect to the partial order $\preccurlyeq$. Categories satisfying this condition are called \textit{partially ordered categories} (or \textit{weakly directed categories}, \textit{directed categories} in the literature by some authors), which is a natural generalization of posets.

Suppose that $\calI$ is a partially ordered category. Then $\calP_i = \Mor(\calI) \backslash \AU i$ is a prime ideal for $i \in \Ob(\calI)$. It is clear that $\calI/\calP_i$ is the full subcategory of $\calI$ with one object $i$. Therefore, we have canonical inclusion functors $\iota_i: \calI_i \to \calI$, $\jmath_i: \calI_i \to \calI/\calP_i$ and $\iota_{\calP_i}: \calI/\calP_i \to \calI$, where $\calI_i$ is the subcategory of $\calI$ consisting of the single object $i$ and the single identity $e_i$. Clearly, $\iota_i = \iota_{\calP_i} \circ \jmath_i$. If $\scrD_i$ satisfies the axiom $\sf{AB3}$ and $\scrD_\gamma$ preserves small coproducts for $\gamma \in \AU i$, then we have the following functors
\[
\xymatrix{
\scrD \lRep \ar[r]^-{\C_{\calP_i}} & (\scrD \circ \iota_{\calP_i}) \lRep \ar[r]^-{\jmath_i^{\ast}} & (\scrD \circ \iota_i) \lRep \cong  \scrD_i.
}
\]
Here $\jmath_i^{\ast}$ is the restriction functor with respect to $\jmath_i$;
see Lemma \ref{l restriction functor along G}.
Denote by $\C_i$ their composite.
Explicitly, given $M \in \scrD \lRep$, one has
\[
\C_i(M) = M_i / \sum_{\theta \in \calP_i(\bullet,i)} M_{\theta} (\scrD_{\theta} (M_{s(\theta)})).
\]

Next we give the right adjoint of $\C_i$. The right adjoint of $\C_{\calP_i}$ is the functor
\[
\lif^{\calP_i}: (\scrD \circ \iota_{\calP_i}) \lRep \to \scrD \lRep;
\]
see Theorem \ref{coker adjoint pair}.
If $\scrD_i$ satisfies the axiom $\sf{AB3}^*$, then by right Kan extension, we can construct the right adjoint
\[
\ran_i: (\scrD \circ \iota_i) \lRep \to (\scrD \circ \iota_{\calP_i}) \lRep,
\]
of $\jmath_i^{\ast}$ which sends an object $K$ in $\scrD_i \cong (\scrD \circ \iota_i) \lRep $ to the object
\[
\ran_i(K) = \prod_{\gamma \in \AU i} \scrD_\gamma(K)
\]
in $(\scrD \circ \iota_{\calP_i}) \lRep$. Define the \emph{stalk functor} at $i$ to be the composite
\[
\xymatrix{
\sta^i:  \scrD_i \cong (\scrD \circ \iota_i) \lRep \ar[r]^-{\ran_i} & (\scrD \circ \iota_{\calP_i}) \lRep \ar[r]^-{\lif^{\calP_i}} & \scrD \lRep.
}
\]
Explicitly, given an object $M_i \in \scrD_i$, for $j \in \Ob(\calI)$, one has
\[
(\sta^i(M_i))_j =
\begin{cases}
0, & \text{if } i \neq j;\\
\prod_{\gamma \in \AU i} \scrD_{\gamma} (M_i), & \text{otherwise}.
\end{cases}
\]

\begin{rmk}
We have defined two cokernel functors $\C_i$ and $\C_{\calP_i}$. For an object $M$ in $\scrD \lRep$, $\C_{\calP_i}(M)$ is an object in $(\scrD \circ \iota_{\calP_i}) \lRep$, which can be also identified with an object in $\scrD_i$. With this identification, one can see that $\C_i(M) = \C_{\calP_i}(M)_i$. On the other hand, note that an object $M$ in $(\scrD \circ \iota_{\calP_i}) \lRep$ is uniquely determined by an object $M_i \in \scrD_i$ equipped with the actions of endomorphisms of the object $i$. Applying the functor $\lif^{\calP_i}$ one has
\[
(\lif^{\calP_i}(M_i))_j =
\begin{cases}
0, & \text{if } i \neq j;\\
M_i, & \text{otherwise},
\end{cases}
\]
which is different from $\sta^i(M_i)$. The reason is very subtle: when applying $\lif^{\calP_i}$ we regard $M_i$ as a representation over $\scrD \circ \iota_{\calP_i}$, while applying $\sta^i$ we only regard $M_i$ as an object in $\scrD_i$ and hence forget the actions of all non-identity morphisms $\theta: i \to i$ on it.
\end{rmk}

The following result follows immediately from Theorem \ref{coker adjoint pair}.

\begin{cor} \label{left adjoint of l stalk}
Let $i$ be an object in $\Ob(\calI)$. Suppose that $\calI$ is a partially ordered category, $\scrD_i$ satisfies both the axioms $\sf{AB3}$ and $\sf{AB3}^*$, and $\scrD_\gamma$ preserves small coproducts for $\gamma \in \AU i$. Then $(\C_i, \sta^i)$ is an adjoint pair.
\end{cor}

\begin{exa}
The stalk functor and its left adjoint have been widely used in literature. In particular, when $\calI$ is the free category associated to a quiver $Q = (Q_0, Q_1)$ without oriented cycles, these two functors have very simple descriptions. For instance, if $\scrD$ is a trivial $\calI$-diagram over a fixed abelian category $\mathcal{A}$ and $V: \calI \to \mathcal{A}$ is a functor, then $\C_i(V)$ is the quotient $V_i / \sum_{\theta \in Q_1(\bullet, i)} V_{\theta}(V_\bullet)$, where $Q_1(\bullet, i)$ is the set of arrows with target $i$.
\end{exa}

\begin{rmk} \label{stalk is exact}
If $\scrD \circ \iota_{\calP_i}$ is exact and $\scrD_i$ satisfies the axiom $\sf{AB4}^*$, then $\ran_i$ is exact as well. Consequently, $\sta^i$ is exact since so is ${\lif}^{\calP_i}$.
\end{rmk}

In the rest of this subsection, we briefly describe some dual constructions and results under the assumption that $\scrD$ admits enough right adjoints.
For $\alpha: i \to j \in \Mor(\calI)$,
denote by
\begin{itemize}
\item $\scrD_\alpha^*$ the right adjoint of $\scrD_\alpha$;

\item $M_\alpha^* : M_i \to \scrD_\alpha^*(M_j)$ the adjoint morphism of the structural morphism $M_\alpha : \scrD_\alpha(M_i) \to M_j$.
\end{itemize}

Let $\calP$ be a prime ideal of $\calI$. For $i \in \Ob(\calI)$, denote
\[
\calP(i, \bullet) = \{\theta \in \Mor(\calI)\, \mid \, \theta \in \calP \text{ with source } i\}.
\]
Suppose that $\scrD_i$ satisfies the axiom $\sf{AB3}^*$ for $i \in \Ob(\calI/\calP)$. By the universal property of limits, there exists a unique morphism ${\psi}_i^{M}$ such that the diagram
\begin{equation} \label{psi}
\xymatrix{
M_i \ar[dr]_{M_\theta^*} \ar[r]^-{{\psi}_i^{M}} &
\lim_{\theta \in \calP(i, \bullet)} \scrD_\theta^*(M_{t(\theta)}) \ar[d] \\
& \scrD_\theta^*(M_{t(\theta)})
}
\end{equation}
in $\scrD_i$ commutes. This commutative diagram enables us to define $\K_\calP : \scrD\lRep \to (\scrD \circ \iota_\calP)\lRep$, which is dual in some sense to $\C_\calP$. One obtains the following result dual to Theorem \ref{coker adjoint pair}.

\begin{thm} \label{ker adjoint pair}
Let $i$ be an object in $\Ob(\calI)$.
Suppose that $\scrD$ admits enough right adjoints and $\scrD_i$ satisfies the axiom $\sf{AB3}^*$ for $i \in \Ob(\calI/\calP)$. Then $\K_\calP$ given above is a functor, and is the right adjoint of ${\lif}^{\calP}$.
\end{thm}

Suppose that $\calI$ is a partially ordered category and $\scrD_i$ satisfies the axiom $\sf{AB3}$. Since $\scrD_\alpha$ preserves small coproducts for $\alpha \in \Mor(\calI)$, one obtains an induction functor
\[
(\jmath_i)_!: \scrD_i \cong (\scrD \circ \iota_i) \lRep \longrightarrow (\scrD \circ \iota_{\calP_i}) \lRep
\]
with respect to the inclusion $\jmath_i: \calI_i \to \calI/\calP_i$; see Proposition \ref{l adjoint along G}. Note that the functor $(\jmath_i)_!$ sends an object $K$ in $\scrD_i$ to the object
\[
(\jmath_i)_!(K) = \coprod_{\gamma \in \AU i} \scrD_\gamma(K_i)
\]
in $(\scrD \circ \iota_{\calP_i}) \lRep$. In this case, another \emph{stalk functor} at $i$ is defined to be the composite
\[
\sta_i: \scrD_i \cong (\scrD \circ \iota_i) \lRep \overset{(\jmath_i)_!} \longrightarrow (\scrD \circ \iota_{\calP_i}) \lRep \overset{\lif^{\calP_i}} \longrightarrow \scrD \lRep.
\]
On the other hand, if $\scrD_i$ satisfies the axiom $\sf{AB3}^*$, we define $\K_i$ to be the composite
\[
\K_i: \scrD \lRep \overset{\K_{\calP_i}} \longrightarrow (\scrD \circ \iota_{\calP_i}) \lRep \overset{\jmath_i^*} \longrightarrow (\scrD \circ \iota_i) \lRep \cong \scrD_i.
\]
An immediate consequence of Theorem \ref{ker adjoint pair} is the following result.

\begin{cor} \label{right adjoint of l stalk}
Let $i$ be an object in $\Ob(\calI)$. Suppose that $\calI$ is a partially ordered category, $\scrD$ admits enough right adjoints and $\scrD_i$ satisfies both axioms $\sf{AB3}$ and $\sf{AB3}^*$. Then $(\sta_i, \K_i)$ is an adjoint pair.
\end{cor}

\section{Characterizations of special homological objects}
\label{Classifications of special homological objects}
\noindent
In this section we introduce a few index categories $\calI$ satisfying certain combinatorial conditions, and characterize special homological objects in $\scrD \lRep$.

\subsection{Left rooted categories and direct categories} \label{rooted categories}

In this subsection we introduce left rooted categories and direct categories, which are partially ordered categories satisfying extra combinatorial conditions. We begin by introducing the following transfinite sequence inspired by the work \cite{EOT04} of Enochs, Oyonarte and Torrecillas.

\begin{ipg}\label{construct}
Suppose that $\calI$ is a partially ordered category, that is, the relation $\preccurlyeq$ on $\Mor(\calI)$ defined by setting $i \preccurlyeq j$ if $\Hom[\calI]i{j} \neq \emptyset$ is a partial order. Define a transfinite sequence $\{V_{\chi}\}_{\chi \, \mathrm{ordinal}}$ of subsets of $\Ob(\calI)$ as follows:
\begin{itemize}
\item for the first ordinal $\chi = 0$, set $V_0 = \emptyset$;

\item for a successor ordinal $\chi+1$, set
 \[
V_{\chi+1} = \left \{ i \in \Ob(\calI) \:
      \left|
        \begin{array}{c}
          i \text{ is not the target of any } \alpha \in \Mor(\calI) \\
          \text{with source } s(\alpha)\neq i \text{ and } s(\alpha) \notin \cup_{\mu \leqslant \chi} V_{\mu}
        \end{array}
      \right.
    \right\};
\]

\item for a limit ordinal $\chi$, set $V_{\chi} = \cup_{\mu < \chi} V_{\mu}$.
\end{itemize}
\end{ipg}

\begin{rmk} \label{rooted cate}
Using the partial order $\preccurlyeq$ defined before, we can give the slightly mysterious definition of the above transfinite sequences a more transparent interpretation. That is, $V_0 = \emptyset$, $V_1$ consists of objects in $\Ob(\calI)$ which is minimal with respect to the partial order $\preccurlyeq$, and $V_{\chi + 1}$ is the union of $V_{\chi}$ and the set of minimal elements in $\Ob(\calI) \backslash V_{\chi}$. By this observation, one deduces a chain $V_1 \subseteq V_2 \subseteq \cdots \subseteq \Ob(\calI)$.
\end{rmk}

\begin{rmk} \label{V forms an ideal}
Recall that a nonempty subset $S$ of a poset $(P, \preccurlyeq)$ is said to be an \emph{ideal} if the following is true: for $j \in S$ and $i \in P$, if $i \preccurlyeq j$, then $i \in S$ as well. In the case that $\calI$ is a partially ordered category, it is easy to check that each $V_{\chi}$ in \ref{construct} with $\chi \geqslant 1$ is an ideal of the poset $(\Ob(\calI), \preccurlyeq)$. Furthermore, if $\chi$ is a successor ordinal, $i \prec j$, and $j \in V_{\chi}$, then $i \in V_{\chi -1}$.
\end{rmk}

Now we are ready to define left (right) rooted categories.

\begin{dfn} \label{df of rooted category}
A partially ordered category $\calI$ is said to be \emph{left rooted} if there exists an ordinal $\zeta$ such that $V_{\zeta} = \Ob(\calI)$.
We say that $\calI$ is \emph{right rooted} if $\calI\op$ is left rooted.
\end{dfn}

\begin{rmk}\label{rmk3.6}
A partially ordered category $\calI$ is left rooted if and only if the poset $(\Ob(\calI), \preccurlyeq)$ is \textit{artinian}, that is, the poset has no infinite descending chain. If $\calI$ is the free category associated to a quiver $Q$ without loops or oriented cycles, then $Q$ is a left rooted quiver in the sense of \cite{EOT04} if and only if $\calI$ is a left rooted category.
\end{rmk}

\textit{Direct categories} are special left rooted categories, which play a prominent role in this paper for classifying special homological objects. The following definition is taken from \cite[Definition 5.1.1]{Ho99}.

\begin{dfn} \label{dfn of direct categories}
A skeletal small category $\calI$ is called a \textit{direct category} if there exists a functor $F: \calI \to \zeta$, where $\zeta$ is an ordinal (viewed as a category in a natural way) such that $F$ sends non-identity morphisms in $\calI$ to non-identity morphisms in $\zeta$. We say that $\calI$ is an \emph{inverse} category if $\calI\op$ is a direct category.
\end{dfn}

A small category $\calI$ is called \emph{locally trivial} if $\AU i$ contains only the identity morphism for all $i \in \Ob(\calI)$, which have been studied in \cite{stein}. It is easy to see that locally trivial categories are partially ordered categories.

The following result clarifies relations among notions introduced above.

\begin{prp} \label{direct categories}
A small skeletal category $\calI$ is a direct category if and only if it is a locally trivial and left rooted category.
\end{prp}

\begin{prf*}
Suppose that $\calI$ is a direct category and let $F: \calI \to \zeta$ be the functor in Definition \ref{dfn of direct categories}. If there exists $i \in \Ob(\calI)$ and a non-identity morphism $f: i \to i$, then $F(i) < F(i)$ by the definition of direct categories, which is absurd. Thus for every $i \in \Ob(\calI)$, there exists no non-identity morphisms from $i$ to itself, that is, $\calI$ is locally trivial. Furthermore, by defining
$V_{\chi} = \{ i \in \Ob(\calI) \mid F(i) \leqslant \chi \}$, one can construct the desired transfinite sequence described in \ref{construct}. Clearly, one has $V_{\zeta} = \Ob(\calI)$, so $\calI$ is left rooted.

Conversely, suppose that $\calI$ is a locally trivial and left rooted category. Then there exists an ordinal $\zeta$ and a transfinite sequence $\{ V_{\chi} \}_{\chi \, \mathrm{ordinal}}$ such that $V_{\zeta} = \Ob(\calI)$.
Now we define another sequence of subsets of $\Ob(\calI)$ as follows: $U_0 = \emptyset$, and
\[
U_{\chi} = V_{\chi} \backslash (\bigcup_{\mu < \chi} V_{\mu}).
\]
Equivalently, $U_{\chi + 1} = V_{\chi+1} \backslash V_{\chi}$, and $U_{\chi} = \emptyset$ if $\chi$ is a limit ordinal. Then one has
\[
\Ob(\calI) = \bigsqcup_{\chi \leqslant \zeta} U_{\chi}.
\]
Note that objects in each $U_{\chi}$ are disjoint, that is, there is no morphisms among these objects. Now define a functor $F: \calI \to \zeta$ such that $F(i) = \chi_i$ for $i \in U_{\chi_i}$ and $F$ sends a morphism $i \to j$ to the unique morphism $\chi_i \leqslant \chi_j$. It is not hard to check that $F$ is well defined. Furthermore, since $\calI$ is locally trivial, $F$ maps non-identity morphisms to non-identity morphisms. Thus $\calI$ is a direct category.
\end{prf*}

\begin{rmk}
The above result on left rooted and direct categories has a dual version for right rooted categories and inverse categories.
\end{rmk}

We end this subsection by giving a few examples appearing frequently in combinatorics, representation theory, and algebraic topology.

\begin{exa} \label{exa tri-root 1}
Any artinian poset is a direct category. For instance, the poset of positive integers and division is an artinian poset. The semi-simplicial category $\Delta_+$ with objects $[n] = \{1, 2, \ldots, n\}$ for $n \geqslant 1$ and morphisms order-preserving injections is a direct category.
\end{exa}

\begin{exa} \label{exa tri-root 2}
Skeletal full subcategories of the category of finite sets and injections is a left rooted category. Similarly, let $A$ be an associative ring. Then skeletal full subcategories of the category of free left $A$-modules of finite rank and $A$-linear injections is left rooted. A skeleton of the category of finitely generated left $A$-modules and injective module homomorphisms is a left rooted category if and only if $A$ is left artinian.
\end{exa}

\subsection{Characterizations of projective and injective objects} \label{projectives and injectives}
In this subsection we apply the functors constructed in Section
\ref{Sec: The evaluation functor and its adjunctions} to characterize projective and injective objects in $\scrD \lRep$.

\begin{setup*}
Throughout this subsection, let $\calI$ be a partially ordered category, and $(\scrD, \eta, \tau)$ be a right exact $\calI$-diagram of Grothendieck categories admitting enough projectives.
\end{setup*}

We mention that $\calP_i = \Mor(\calI) \backslash \AU i$ is a prime ideal of $\Mor(\calI)$ for each $i \in \Ob(\calI)$, and there is an inclusion functor $\iota_{\calP_i}: \calI/\calP_i \to \calI$, which induces a restriction functor $\iota_{\calP_i}^{\ast}: \scrD \lRep \to (\scrD \circ \iota_{\calP_i}) \lRep$ and an induction functor $(\iota_{\calP_i})_!: (\scrD \circ \iota_{\calP_i}) \lRep \to \scrD \lRep$.
The following lemma is straightforward. For definitions of $\C_i$ and $\C_{\calP_i}$, please refer to Subsection \ref{Sec: The stalk functor and its adjunctions}.

\begin{lem} \label{cor is projective}
Let $i$ be an object in $\Ob(\calI)$ and suppose that $\scrD_\gamma$ preserves small coproducts for $\gamma \in \AU i$. If $P$ is a projective object in $\scrD \lRep$, then $\C_{\calP_i}(P)$ is a projective object in $(\scrD \circ \iota_{\calP_i}) \lRep$. If furthermore $\scrD_i$ satisfies the axiom $\sf{AB4}^*$ and $\scrD \circ \iota_{\calP_i}$ is exact, then $\C_i(P)$ is a projective object in $\scrD_i$.
\end{lem}

\begin{prf*}
It follows from Theorem \ref{coker adjoint pair} that $(\C_{\calP_i}, \lif^{\calP_i})$ is an adjoint pair, and the functor $\lif^{\calP_i}$ is exact. Thus the first conclusion holds. By Corollary \ref{left adjoint of l stalk}, $(\C_i, \sta^i)$ is an adjoint pair, and the functor $\sta^i$ is exact whenever $\scrD_i$ satisfies the axiom $\sf{AB4}^*$ and $\scrD \circ \iota_{\calP_i}$ is exact; see Remark \ref{stalk is exact}. Thus the second conclusion holds.
\end{prf*}

Given a family $\calX = \{\calX_i\}_{i \in \Ob(\calI)}$ with each $\calX_i$ a subcategory of $\scrD_i$, we use the functor $\C_\bullet$ to define a subcategory $\Phi(\calX)$ of $\scrD \lRep$. This subcategory will play a key role for us to characterize homological objects in $\scrD \lRep$ and to describe the cofibrant objects in the abelian model structures on $\scrD \lRep$ in our next paper.

\begin{dfn} \label{s and phi}
Suppose that $\scrD_\alpha$ preserves small coproducts for $\alpha \in \Mor(\calI)$. Define a subcategory of $\scrD \lRep$:
\[
\Phi(\calX)= \{ X \in \scrD \lRep \mid \varphi_i^X \text{ is a monomorphism and } \C_i(X) \in {\calX}_i \text{ for each } i \in \Ob(\calI) \},
\]
where $\varphi_i^X: \colim_{\theta \in \calP_i(\bullet, i)} \scrD_\theta(X_{s(\theta)})\to X_i$ is given in (\ref{map-def}) and $\C_i(X)=\coker(\varphi_i^X)$. In particular,
\[
\Phi(\scrD) = \{ X \in \scrD \lRep \mid \varphi_i^X \text{ is a monomorphism for each } i \in \Ob(\calI) \}.
\]
\end{dfn}

\begin{lem}\label{lem3.15}
Suppose that $\scrD_\alpha$ preserves small coproducts for $\alpha \in \Mor(\calI)$. Then $\fre_j(M_j)$ is in $\Phi(\scrD)$ for each $j \in \Ob(\calI)$ and all objects $M_j$ in $\scrD_j$.
\end{lem}

\begin{prf*}
Fix $j\in\Ob(\calI)$ and let $M_j$ be an object in $\scrD_j$. Denote $\fre_j(M_j)$ by $L$ for brevity. It is clear that $\varphi_j^L$ is a monomorpism, so it suffices to show that $\varphi_i^L$ is also a monomorphism for each $i \in \Ob(\calI)$ with $i\neq j$. We prove a stronger result, that is, the morphism
\begin{equation*}
\varphi_i^L: \colim_{\substack{\sigma: \, t \to i \\ t \neq i}} \scrD_{\sigma} (L_t) = \colim_{\substack{\sigma: \, t \to i \\ t \neq i}} \scrD_{\sigma} (\coprod_{\theta: \, j \to t} \scrD_{\theta}(M_j)) \longrightarrow \coprod_{\alpha: \, j \to i} \scrD_\alpha(M_j) = L_i
\end{equation*}
is actually an isomorphism for each $i \in \Ob(\calI)$ with $i\neq j$, where we write
\[
\colim_{\sigma \in \calP_i(\bullet, i)} = \colim_{\substack{\sigma: \, t \to i \\ t \neq i}}
\]
which looks more intuitive. By the universal property of colimits, we only need to show that $L_i$ is isomorphic to the colimit on the left side.

Denote by $\calI_{\prec i}$ the full subcategory of $\calI$ consisting of objects $t$ such that $t \prec i$. Clearly, the object $j$ is contained in $\calI_{\prec i}$. We have two inclusion functors
$F: \calI_j \to \calI_{\prec i}$ and $G: \calI_{\prec i} \to \calI$, whose composite is precisely $\iota_j: \calI_j \to \calI$. Consequently, one has $\fre_j \cong G_!F_!$, where $G_!$ and $F_!$ are the corresponding induction functors. Note that for an object $K$ in $(\scrD \circ G) \lRep$,
\[
(G_!(K))_i = \colim_{\sigma \in G/i} \scrD_{\sigma} (K_{s(\sigma)}) = \colim_{\substack{\sigma: \, t \to i\\ t \neq i}} \scrD_{\sigma} (K_t).
\]
Therefore, for an object $M_j$ in $\scrD_j$,
\[
\coprod_{\alpha: \, j \to i} \scrD_{\alpha} (M_j) = (\fre_j(M_j))_i \cong (G_!(F_!(M_j)))_i = \colim_{\substack{\sigma: \, t \to i\\ t \neq i}} \scrD_{\sigma} \big{(}(F_!(M_j))_t \big{)}.
\]
But one also has
\[
(F_!(M_j))_t = \coprod_{\theta: \, j \to t} \scrD_{\theta} (M_j)
\]
since $F_!$ is induced by the inclusion $\calI_j \to \calI_{\prec i}$. Consequently,
\[
\coprod_{\alpha: \, j \to i} \scrD_{\alpha} (M_j) \cong \colim_{\substack{\sigma: \, t \to i\\ t \neq i}} \scrD_{\sigma} \big{(} \coprod_{\theta: \, j \to t} \scrD_{\theta} (M_j) \big{)}
\]
as claimed.
\end{prf*}

The next result can be checked routinely.

\begin{lem}\label{phi sum and summand}
Suppose that $\scrD_\alpha$ preserves small coproducts for $\alpha \in \Mor (\calI)$. Then $\Phi(\scrD)$ is closed under direct summands and small coproducts.
\end{lem}

The following result gives a characterization of projective objects in $\scrD \lRep$.

\begin{thm} \label{on side for porjecyives}
Suppose that $\scrD_\alpha$ preserves small coproducts for $\alpha \in \Mor(\calI)$. An object $H$ in $\scrD \lRep$ is projective only if the following conditions hold for every $j \in \Ob(\calI)$:
\begin{prt}
\item ${\varphi}_{j}^{H}: \colim_{\sigma \in \calP_j(\bullet, j)} \scrD_\sigma(H_{s(\sigma)}) \to H_j$ is a monomorphism and
\item $\C_{\calP_j} (H)$ is a projective object in $(\scrD \circ \iota_{\calP_j}) \lRep$.
\end{prt}
If furthermore $\calI$ is left rooted, then the converse is true.
\end{thm}

\begin{prf*}
Note that any projective object in $\scrD \lRep$ is isomorphic to a direct summand of a direct sum of members in the family $\{\fre_i(P_i) \mid  i \in \Ob(\calI)\ \mathrm{ and } \ P_i \in \Prj{\scrD_i} \}$; see Remark \ref{generator for R-represetations}. Thus statement (a) follows from Lemmas \ref{lem3.15} and \ref{phi sum and summand}, and statement (b) holds by Lemma \ref{cor is projective}.

Suppose that $\calI$ is a left rooted category. Let $H$ be an object in $\scrD \lRep$ satisfying statements (a) and (b), and let $\{V_{\chi}\}_{\chi \, \mathrm{ordinal}}$ be the transfinite sequence of subsets of $\Ob(\calI)$ defined in \ref{construct}. Then $\Ob(\calI) = V_{\zeta}$ for a certain ordinal $\zeta$. We construct a family $\{ H^{\chi} \}_{\chi \leqslant \zeta}$ of projective subobjects of $H$ as follows.

Set $H^0 = 0$ and $H^1 = \coprod_{i \in V_1} (\iota_{\calP_i})_!(H_i)$. Note that objects in $V_1$ are minimal. Consequently, for $i \in V_1$, one has $H_i = \C_{\calP_i} (H)$ which is projective in $(\scrD \circ \iota_{\calP_i}) \lRep$. Since $(\iota_{\calP_i})_!$ is the left adjoint of the exact functor $\iota_{\calP_i}^{\ast}$ by Proposition \ref{l adjoint along G}, it preserves projective objects. Therefore, $(\iota_{\calP_i})_!(H_i)$ is projective in $\scrD \lRep$, and hence $H^1$ is also projective in $\scrD \lRep$.

For $\chi > 1$, there are two cases:
\begin{prt}
\item $\chi$ is a successor ordinal. Define
\[
H^{\chi} = \coprod_{i \in V_{\chi} \backslash V_{\chi-1}} (\iota_{\calP_i})_!(\C_{\calP_i}(H)),
\]
which is clearly is projective in $\scrD \lRep$ by statement (b).

\item $\lambda$ is a limit ordinal and the projective objects $H^\kappa$ have been constructed for all ordinals $\kappa < \lambda$. Define $H^\lambda = \coprod_{\kappa < \lambda} H^\kappa$, which is projective.
\end{prt}

Now define $H^\zeta = \coprod_{\chi < \zeta} H^\chi$. It is routine to check that $H$ is indeed isomorphic to $H^\zeta$, using the following isomorphism
\[
H_i \cong (\colim_{\theta \in \calP(\bullet, i)} \scrD_{\theta} (H_{s(\theta)})) \oplus \C_{\calP_i}(H).
\]
Thus $H$ is a projective object in $\scrR \lRep$, as desired.
\end{prf*}

\begin{exa}
In Theorem \ref{on side for porjecyives} one cannot replace the functor $\C_{\calP_i}$ by the functor $\C_i$. Here we give a trivial example to illustrate the difference. Let $\calI$ be a cyclic group of order 2 which is viewed as a category with one object $i$ and two morphisms, and let $\scrD$ be the trivial diagram with $\scrD_i$ the category of vector spaces over a field $k$ of characteristic 2. Clearly, an object in $\scrD \lRep$ is nothing but a representation of $\calI$. Let $k$ be the trivial representation. Then it clearly lies in $\Phi(\sf Proj_{\bullet})$, but is not a projective representation of $\calI$. The reason is obvious: although an object $M_i$ in $(\scrD \circ \iota_{\calP_i}) \lRep$ can be viewed as an object in $\scrD_i$, the condition that $M_i$ is projective in $\scrD_i$ cannot guarantee that $M_i$ is also projective in $(\scrD \circ \iota_{\calP_i}) \lRep$.
\end{exa}

An $\calI$-diagram $\scrD$ of abelian categories is called \textit{locally exact} if for any $i \in \Ob(\calI)$ and any endomorphism from $i$ to $i$, the functor $\scrD_{\gamma}$ is exact. In particular, if $\calI$ is locally trivial (for instances, direct or inverse), then $\scrD$ is locally exact.

Denote by $\sf Proj_{\bullet}$ the family $\{\Prj{\scrD_i}\}_{i \in \Ob(\calI)}$. Then one has:

\begin{cor} \label{proj for direct categories}
Suppose that $\scrD$ is locally exact, $\scrD_i$ satisfies the axiom $\sf{AB4}^*$ for $i \in \Ob(\calI)$, and $\scrD_\alpha$ preserves small coproducts for $\alpha \in \Mor(\calI)$. Then
\[
\Prj{\scrD \lRep} \subseteq \Phi(\sf Proj_{\bullet}).
\]
If furthermore $\calI$ is a direct category, then these two categories coincide.
\end{cor}

\begin{proof}
Note that $\scrD \circ \iota_{\calP_i}$ is exact for every $i \in \Ob(\calI)$ as $\scrD$ is locally exact by assumption. By replacing $\C_{\calP_i}$ with $\C_i$, a similar version of statement (b) in the previous theorem holds by Lemma \ref{cor is projective}. Now the same proof in the previous theorem shows that every projective object in $\scrD \lRep$ also satisfies statement (a). This establishes the first conclusion. The second one is trivial as in this case $\calI/\calP_i = \calI_i$ which implies that $(\scrD \circ \iota_{\calP_i}) \lRep \cong \scrD_i$, $\C_i = \C_{\calP_i}$ and $\fre^i = (\iota_{\calP_i})_!$.
\end{proof}

Suppose that $\scrD$ admits enough right adjoints, and denote by $\scrD_\alpha^*$ the right adjoint of $\scrD_\alpha$ for $\alpha \in \Mor(\calI)$. We give characterizations of injective objects in $\scrD \lRep$, whose proof is dual to that of Theorem \ref{on side for porjecyives}.

\begin{thm} \label{on side for injectives}
Suppose that $\scrD$ admits enough right adjoints and $\scrD_i$ satisfies the axiom $\sf{AB4}^*$ for $i \in \Ob(\calI)$. An object $I$ in $\scrD \lRep$ is injective only if the following conditions hold for every $j \in \Ob(\calI)$:
\begin{prt}
\item $\psi_j^I: I_j \to \lim_{\sigma \in \calP_j(j, \bullet)} \scrD_\sigma^*(I_{t(\sigma)})$ is an epimorphism and
\item $\K_{\calP_j}(I)$ is an injective object in $(\scrD_j \circ \iota_{\calP_j}) \lRep$.
\end{prt}
If further $\calI$ is right rooted, then the converse is true.
\end{thm}

\begin{dfn} \label{s and psi}
Let $\calY = \{\calY_i\}_{i \in \Ob(\calI)}$ with each member a subcategory of $\scrD_i$. Suppose that $\scrD$ admits enough right adjoints. Define
\[
\Psi(\calY) = \{ Y \in {\scrD \lRep} \mid {\psi_i^Y} \text{ is an epimorphism and } \K_i(Y) \in {\calY}_i \text{ for } i \in \Ob(\calI) \},
\]
where $\psi_i^Y: Y_i \to \lim_{\sigma \in \calP_i(i, \bullet)} \scrD_\sigma^*(Y_{t(\sigma)})$ is given in (\ref{psi}) and $\K_i(Y)=\ker(\psi_i^Y)$.
\end{dfn}

Denote by $\sf Inj_{\bullet}$ the family $\{\Inj{\scrD_i}\}_{i \in \Ob(\calI)}$. The dual result of Corollary \ref{proj for direct categories} is:

\begin{cor} \label{inj for direct categories}
Suppose that $\scrD$ is locally exact, admits enough right adjoints and $\scrD_i$ satisfies the axiom $\sf{AB4}^*$ for $i \in \Ob(\calI)$. Then
\[
\Inj{\scrD \lRep} \subseteq \Psi(\sf Inj_{\bullet}).
\]
If further $\calI$ is an inverse category, then these two categories coincide.
\end{cor}

\begin{rmk}
For the trivial diagram indexed by the free category associated to a left rooted quiver, Corollary \ref{proj for direct categories} was proved in \cite[Theorem 3.1]{EE05}. For the trivial diagram indexed by the free category associated to a right rooted quiver, Corollary \ref{inj for direct categories} was proved in \cite[Theorem 4.2]{EER09}.
\end{rmk}

\begin{rmk} \label{another explain for pro in}
By the proofs of Theorem \ref{on side for porjecyives} and Corollary \ref{proj for direct categories}, if $\calI$ is a direct category, then every projective object $P$ in $\scrD \lRep$ is isomorphic to $\coprod_{i \in \Ob(\calI)} \fre_i(\C_i(P))$. Dually, if $\calI$ is an inverse category, then every injective object $I$ in $\scrD \lRep$ is isomorphic to $\prod_{i \in \Ob(\calI)} \fre_i (\K_i(I))$.
\end{rmk}

\appendix
\section*{Appendix. Modules over presheaves of associative rings}
\stepcounter{section}

\noindent
Estrada and Virili introduced in \cite{SS2017} the notions of representations $R$ of small categories $\calI$ on $\Cat$ (the 2-category of small preadditive categories) and modules over $R$, and established a few fundamental homological facts on the category of these modules. We observe that a representation $R$ of $\calI$ on $\Cat$ in that paper is precisely an $\calI$-diagram of small preadditive categories. In this Appendix, we explore relations between modules over $R$ in their sense and representations over diagrams in our sense. We show that an $\calI$-diagram $R$ of associative rings induces a right exact $\calI$-diagram $\overline{\scrR}$ of module categories such that the category $R\lMod$ of left $R$-modules in \cite{SS2017} coincides with the category $\overline{\scrR} \lRep$. We mention that this result can be extended to $\calI$-diagrams of small preadditive categories (or a representation of $\calI$ on $\Cat$). Therefore, the work described in \cite{SS2017} also falls into our framework.

The following definition is taken from \dfncite[3.6]{SS2017}.

\begin{dfn}\label{Estrada and Virili}
Let $(R, \eta, \tau)$ be an $\calI$-diagram of associative rings. A left $R$-module $M$ consists of the following data:
\begin{itemize}
\item for $i \in \Ob(\calI)$, a left $R_i$-module $M_i$, and
\item for $\alpha: i \to j \in \Mor(\calI)$, a morphism $M_\alpha^*: M_i \to (R_\alpha)^*(M_j)$ in $R_i \lMod$, where $(R_\alpha)^* = \Hom[R_j]{R_j}{-}$ is the restriction of scalars along $R_\alpha$,
\end{itemize}
such that the following axioms are satisfied:

\noindent ({Mod.1}) Given composable morphisms $i \overset{\alpha} \to j \overset{\beta} \to k \in \Mor(\calI)$, there is a commutative diagram
\[
\xymatrix{
M_i \ar[rr]^-{M_{\beta\alpha}^*} \ar[d]_{M_\alpha^*} & & (R_{\beta\alpha})^*(M_k) \ar[d]^{\id_{M_k} \ast \tau_{\beta,\alpha}}\\
(R_{\alpha})^*( M_j) \ar[rr]^-{(R_\alpha)^*(M_\beta^*)} & & (R_\alpha)^*((R_\beta)^*(M_k))
}
\]
in $R_i \lMod$.

\noindent ({Mod.2}) For $i \in \Ob(\calI)$, there exists a commutative diagram
\[
\xymatrix{
M_i \ar[rr]^-{\id_{M_i}} \ar[dr]_{M_{e_i}^*} & & M_i\\
 & (R_{e_i})^*(M_i) \ar[ur]_{\id_{M_i} \ast \eta_i}
}
\]
in $R_i \lMod$.

A morphism $\omega: M \to M'$ between two left $R$-modules is a family $\{ \omega_i: M_i \to M'_i\}_{i \in \Ob(\calI)}$ of morphisms such that the diagram
\[
\xymatrix{
M_i \ar[rr]^-{\omega_i} \ar[d]_{M_\alpha^*} & & M'_i \ar[d]^{{M'}_\alpha^*}\\
(R_\alpha)^*(M_j) \ar[rr]^-{(R_\alpha)^*(\omega_j)} & & (R_\alpha)^*(M'_j)
}
\]
in $R_i \lMod$ commutes for any $\alpha: i \to j \in \Mor(\calI)$.
Denote by $R \lMod$ the category of left $R$-modules.
\end{dfn}

Now we construct the right exact $\calI$-diagram ($\overline{\scrR}$, $\overline{\eta}$, $\overline{\tau}$) as follows:
\begin{itemize}
\item for $i \in \Ob(\calI)$, set $\overline{\scrR}_i = R_i \lMod$, the category of left $R_i$-modules;

\item for $\alpha: i \to j \in \Mor(\calI)$, set $\overline{\scrR}_\alpha: R_i \lMod \to R_j \lMod$ to be $R_j \otimes_{R_i} -$;

\item for $i \in \Ob(\calI)$, set $\overline{\eta}_i: \id_{R_i \lMod} \to \overline{\scrR}_{e_i} = R_i \otimes_{R_i} -$ to be the classical natural isomorphism defined by the isomorphism $M \cra R_i \otimes_{R_i} M$ for any left $R_i$-module $M$;

\item for composable morphisms $\alpha: i \to j$ and $\beta: j \to k$ in $\Mor(\calI)$, set $\overline{\tau}_{\beta,\alpha}: \overline{\scrR}_\beta \circ \overline{\scrR}_\alpha \to \overline{\scrR}_{\beta\alpha}$ to be the composite of the following classical natural isomorphisms
\[
R_k \otimes_{R_j} (R_j \otimes_{R_i} -) \cra (R_k \otimes_{R_j} R_j) \otimes_{R_i} - \cra R_k \otimes_{R_i} -.
\]
\end{itemize}
It is routine to show that ($\overline{\scrR}$, $\overline{\eta}$, $\overline{\tau}$) is an $\calI$-diagram of module categories and it is clear that it is right exact. Furthermore, we have:

\begin{thm} \label{EV them}
The category $R \lMod$ coincides with $\overline{\scrR} \lRep$.
\end{thm}

\begin{prf*}
For any $\alpha: i \to j \in \Mor(\calI)$, note that $((R_\alpha)_!, (R_\alpha)^*)$ is an adjoint pair, where $(R_\alpha)_! = R_j \otimes_{R_i} -$. Let $M_\alpha : (R_\alpha)_!(M_i) \to M_j$ be the adjoint morphism of $M_\alpha^* :M_i \to (R_\alpha)^*(M_j)$. Then it is routine to check by Yoneda Lemma that a left $R$-module $M$ is precisely an object in $\overline{\scrR} \lRep$. Furthermore, by Yoneda Lemma again, a morphism between two left $R$-modules is precisely a morphism in $\overline{\scrR} \lRep$.
\end{prf*}

By \thmcite[3.18]{SS2017}, the category $R\lMod$ is a Grothendieck category. If further $\calI$ is a poset, then $R \lMod$ has a projective generator. According to Theorem \ref{EV them}, $R \lMod$ coincides with the category $\overline{\scrR} \lRep$. Applying Theorem \ref{local grenthe induce local grenthen} and Proposition \ref{local induce local} to this special case, we obtain the following result, which improves \thmcite[3.18]{SS2017} by dropping the unessential condition that $\calI$ is a poset.

\begin{cor} \label{right module is grethendieck}
Let $R$ be an $\calI$-diagram of associative rings. Then $R \lMod$ is a locally finitely presented Grothendieck category with a projective generator.
\end{cor}

\begin{rmk}\label{underline R}
An $\calI$-diagram $R$ of associative rings induces also a right exact $\calI\op$-diagram $(\widetilde{\scrR}, \widetilde{\eta}, \widetilde{\tau})$ of module categories as follows:
\begin{itemize}
\item for $i \in \Ob(\calI \op)$, $\widetilde{\scrR}_i = \rMod R_i$, the category of right $R_i$-modules;

\item for $\alpha\op: j \to i \in \Mor(\calI \op)$, $\widetilde{\scrR}_{\alpha \op} = -\otimes_{R_j}R_j: \rMod R_j \to \rMod R_i$;

\item for $i \in \Ob(\calI \op)$, $\widetilde{\eta}_i: \id_{\rMod R_i} \to \widetilde{\scrR}_{e_i} = - \otimes_{R_i} R_i$ is the classical natural isomorphism defined by the isomorphism $N \cra N \otimes_{R_i} R_i$ for any right $R_i$-module $N$;

\item for any pair of composable morphisms $\beta\op : k \to j$ and $\alpha\op : j \to i$ in $\Mor(\calI\op)$,
\[
\widetilde{\tau}_{\alpha\op,\beta\op}: \widetilde{\scrR}_{\alpha \op} \circ \widetilde{\scrR}_{\beta \op} \to \widetilde{\scrR}_{\alpha \op\beta \op}
\]
is the composite of the following classical natural isomorphisms
\[
(- \otimes_{R_k} R_k) \otimes_{R_j} R_j \cra - \otimes_{R_k}(R_k \otimes_{R_j} R_j) \cra - \otimes_{R_k}R_k.
\]
\end{itemize}
It is clear that the $\calI\op$-diagram $\widetilde{\scrR}$ admits enough right adjoints.
Similar to Theorem \ref{EV them}, one can show that the category $\rMod R$ of right $R$-modules in the sense of \cite{SS2017} coincides with $\widetilde{\scrR} \lRep$.
\end{rmk}


\bibliographystyle{amsplain-nodash}


\def\cprime{$'$}
  \providecommand{\arxiv}[2][AC]{\mbox{\href{http://arxiv.org/abs/#2}{\sf
  arXiv:#2 [math.#1]}}}
  \providecommand{\oldarxiv}[2][AC]{\mbox{\href{http://arxiv.org/abs/math/#2}{\sf
  arXiv:math/#2
  [math.#1]}}}\providecommand{\MR}[1]{\mbox{\href{http://www.ams.org/mathscinet-getitem?mr=#1}{#1}}}
  \renewcommand{\MR}[1]{\mbox{\href{http://www.ams.org/mathscinet-getitem?mr=#1}{#1}}}
\providecommand{\bysame}{\leavevmode\hbox to3em{\hrulefill}\thinspace}
\providecommand{\MR}{\relax\ifhmode\unskip\space\fi MR }
\providecommand{\MRhref}[2]{%
  \href{http://www.ams.org/mathscinet-getitem?mr=#1}{#2}
}
\providecommand{\href}[2]{#2}

\end{document}